\documentclass[twoside,reqno,11pt,a4paper]{amsart}
\usepackage{amsmath, amssymb, amsthm}
\usepackage{url}
\usepackage[hidelinks]{hyperref}
\usepackage{mathtools}
\usepackage{upgreek}
\usepackage{environ}
\usepackage{xcolor}
\numberwithin{equation}{section}
\theoremstyle{plain}
\newtheorem{thm}{Theorem}[section]
\newtheorem{lem}[thm]{Lemma}

\newtheorem{cor}[thm]{Corollary}
\newtheorem{prop}[thm]{Proposition}

\theoremstyle{definition}

\theoremstyle{remark}

\DeclareMathOperator{\Gal}{Gal}
\DeclareMathOperator{\Fix}{Fix}
\DeclareMathOperator{\Frob}{Frob}
\newcommand{\calZ}{\mathcal Z}
\renewcommand{\le}{\leqslant}
\renewcommand{\leq}{\leqslant}
\renewcommand{\ge}{\geqslant}
\renewcommand{\geq}{\geqslant}
\begin{document}
\title[Moments of Dedekind Zeta Functions]{Sharp Upper Bounds for Moments of Dedekind Zeta Functions}
    \author{Benjamin Durkan}
\address{\scriptsize Benjamin Durkan, Department of Mathematics, The University of Manchester, Oxford Road, Manchester, M13 9PL}
\email{benjamin.durkan@manchester.ac.uk}
	\author{Nilmoni Karak}
	\address{ \scriptsize Nilmoni Karak, Department of Mathematics,
		Indian Institute of Technology	Kharagpur,
		Kharagpur-721302, India}
	\email{nilmonikarak@gmail.com, nilmonimath@kgpian.iitkgp.ac.in}

	\author{Kamalakshya Mahatab}
	\address{ \scriptsize  Kamalakshya Mahatab, Department of Mathematics,
		Indian Institute of Technology	Kharagpur,
		Kharagpur-721302, India}
	\email{accessing.infinity@gmail.com, kamalakshya@maths.iitkgp.ac.in}
	

\begin{abstract}
	Assuming the Generalised Riemann Hypothesis, we establish conjecturally sharp upper bounds for shifted moments of products of Dedekind zeta functions of arbitrary number fields. This improves results of Milinovich and Turnage-Butterbaugh and extends a recent result of Hagen. As applications, we obtain mean-square bounds for short-interval sums of the coefficients of Dedekind zeta functions and upper bounds for the large deviations of Dedekind zeta functions. Our results apply to both Galois and non-Galois extensions.
\end{abstract}

	\maketitle
	 
\section{Introduction}      
Understanding the behaviour of $L$-functions on the critical line is a central problem in analytic number theory. A fundamental approach to this problem is the study of moments, which have numerous applications to non-vanishing, value distribution, and correlations of $L$-functions. Assuming the Riemann Hypothesis (RH), Soundararajan \cite{Sound} obtained almost sharp upper bounds for all positive moments of the Riemann zeta function. He showed that 
	\begin{equation*}
	\int_{T}^{2T}\left| \zeta\left(\tfrac{1}{2}+it\right)\right|^{2k}\;dt\ll_{k,\varepsilon} T(\log T)^{k^2 +\varepsilon},
	\end{equation*}
	for any $k>0$ and every $\varepsilon>0$.
	Subsequently, Harper \cite{Harper} removed the factor $(\log T)^{\varepsilon}$. Assuming RH, he proved that, for every fixed $k>0$,
	\begin{equation}\label{harper-upper-bound}
\int_{T}^{2T}\left| \zeta\left(\tfrac{1}{2}+it\right)\right|^{2k}\;dt\ll_{k} T(\log T)^{k^2}.
	\end{equation}
The corresponding lower bound is known unconditionally. More precisely, for every fixed real $k>0$,
\begin{equation*}
\int_T^{2T}|\zeta(\tfrac{1}{2}+it)|^{2k}\;dt\gg T(\log T)^{k^2}.
\end{equation*}
For $k>1$, unconditional lower bounds of the conjectured order of magnitude which vary continuously with $k$ were established by Radziwi{\l}{\l} and Soundararajan \cite{Raziwill-Sound}; Heap and Soundararajan \cite{Heap-Sound} later extended the range to $k>0$. Effective unconditional lower bounds can also be obtained (cf. \cite{Conrey-Ghosh-III, DP, Page, Sound-mean-values}).
This leads naturally to shifted moments, which capture correlations between values of $L$-functions at distinct points on the critical line. For the Riemann zeta function, shifted moments have been investigated by several authors. Chandee \cite[Theorem~1.1]{Chandee_shifted} stated that, under RH, for sufficiently large $T$ and every $\varepsilon>0$,
\begin{align*}
	\int_{T}^{2T} \prod_{i=1}^{r}\left| \zeta\left(\tfrac{1}{2}+i(t+b_i)\right)\right|^{a_i}\ll_{\mathbf{a},\varepsilon} T(\log T)^{(a_1^2+\cdots + a_{r}^2)/4+\varepsilon} \prod_{1\le  i<j\le  r}\left(\min\left\{\tfrac{1}{|b_i-b_j|},\log T \right\}\right)^{\tfrac{a_ia_j}{2}},
\end{align*}
	where $\mathbf{a}=(a_1,\ldots,a_r)$ consists of fixed positive real numbers and the distinct shifts $b_i=o(T)$ satisfy $b_i-b_j=O(1)$, and where each of the limits $b_i\log T$ and $(b_i-b_j)\log T$ exists in $\mathbb R\cup\{\pm\infty\}$. Assuming RH, Ng, Shen, and Wong \cite[Theorem~1.3]{Ng-Shen-wong} obtained a sharp two-shift bound for equal factor exponents $2k$, where $k\ge1/2$, under $|b_1|,|b_2|\le T/2$ and $|b_1+b_2|\le T^{0.6}$; their correlation factor is $\min\{|b_1-b_2|^{-1},\log T\}$ for $|b_1-b_2|\le1/100$ and $\log(2+|b_1-b_2|)$ otherwise. Under RH, Curran \cite[Theorem~1.1]{Curran-Correlation} extended the sharp bound to any fixed number of nonnegative factor exponents and shifts $|b_i|\le T/2$, and obtained stronger bounds when the differences $|b_i-b_j|$ are unbounded. More precisely, in the present exponent convention, he proved that
	\begin{equation}\label{curran-upper-bound}
		\int_{T}^{2T} \prod_{i=1}^{r}\left| \zeta\left(\tfrac{1}{2}+i(t+b_i)\right)\right|^{a_i}\ll_{\mathbf{a}} T(\log T)^{(a_1^2+\cdots + a_{r}^2)/4} \prod_{1\le  i< j\le  r}\left| \zeta(1+i(b_i-b_j)+1/\log T)\right|^{\tfrac{a_ia_j}{2}}.
	\end{equation}

	This article focuses on extending the upper bounds \eqref{harper-upper-bound} and \eqref{curran-upper-bound} to moments of the Dedekind zeta functions associated with arbitrary number fields. The Dedekind zeta function associated with the number field $K$ is defined as 
	\begin{equation}\label{dedekind-zeta}
		\zeta_K(s)= \sum_{0\ne\mathfrak{a}\subseteq \mathcal{O}_K} \frac{1}{N(\mathfrak{a})^s} = \prod_{0\ne\mathfrak{p}\subseteq \mathcal{O}_K} \left( 1- \frac{1}{N(\mathfrak{p})^s}\right)^{-1}
	\end{equation}
	for $\Re(s)>1$, where the sum and the product are taken over nonzero integral ideals $\mathfrak{a}$ and nonzero prime ideals $\mathfrak{p}$ of $\mathcal{O}_K$, respectively, and $N(\mathfrak a)=|\mathcal O_K/\mathfrak a|$.  The function $\zeta_K(s)$ extends meromorphically to $\mathbb{C}$, with a simple pole at $s=1$, and satisfies the functional equation $\zeta_K(1-s)= \chi_K(s) \zeta_K(s)$, where \cite[pp. 254-255]{S. Lang}
	\begin{equation*}
		\chi_K(s)= (2(2\pi)^{-s}\Gamma(s))^{[K:\mathbb{Q}]} |D_{K}|^{s-1/2} (\cos \pi s/2)^{r_1+r_2} (\sin \pi s/2)^{r_2}.
	\end{equation*}
	If $(r_1,r_2)$ is the signature of $K$, then the orders of the trivial zeros are $r_2$ at negative odd integers, $r_1+r_2$ at nonzero negative even integers, and $r_1+r_2-1$ at zero. The Generalised Riemann Hypothesis (GRH) for $\zeta_K(s)$ states that all the non-trivial zeros of $\zeta_K(s)$ lie on the critical line $\Re(s)=1/2$. It is well known that the Dedekind zeta function $\zeta_K(s)$ factors as a product of Artin $L$-functions. Specifically, if $K/\mathbb{Q}$ is a finite Galois extension, then
	\begin{equation*}
		\zeta_K(s)=\prod_{\chi\in\widehat{\mathrm{Gal}(K/\mathbb{Q})}}
		L(s,\chi)^{\chi(1)}.
	\end{equation*}
	Here $\widehat{\mathrm{Gal}(K/\mathbb{Q})}$ denotes the set of irreducible characters of $\mathrm{Gal}(K/\mathbb{Q})$. Furthermore,
	\begin{equation*}
		\sum_{\chi\in\widehat{\mathrm{Gal}(K/\mathbb{Q})}}
		\chi(1)^2
		=
		\left|\mathrm{Gal}(K/\mathbb{Q})\right|
		=
		[K:\mathbb{Q}].
	\end{equation*}
Langlands reciprocity predicts automorphy for the irreducible Artin factors, and applying the unitary-family moment heuristic and the recipe of Conrey, Farmer, Keating, Rubinstein, and Snaith \cite{CFKRS-I,CFKRS-II} to this factorisation predicts that, for any fixed real $k>0$,
\begin{equation}\label{conjecture-moment-dedekind}
	\int_{T}^{2T}\left|\zeta_K\!\left(\tfrac12+it\right)\right|^{2k}\;dt
	\sim
	C_{k,K}\,T(\log T)^{[K:\mathbb{Q}]k^2}
\end{equation}
as $T\to\infty$, where $C_{k,K}$ is a positive constant depending only on $k$ and $K$. A detailed discussion of this conjecture appears in the work of Heap \cite{Heap1}. The conjecture \eqref{conjecture-moment-dedekind} is already known for $k=1$ when $K$ is a quadratic extension of $\mathbb{Q}$. Write $K=\mathbb{Q}(\sqrt{D_K})$, where $D_K$ is the fundamental discriminant. Motohashi \cite{Motohasi} showed that
\begin{equation*}
	\int_{T}^{2T} \left| \zeta_K\left(\tfrac{1}{2}+it\right)\right|^{2}\;dt
	\sim
	\frac{6}{\pi^2}L\left(1,\chi_{D_K}\right)^2
	\prod_{p\mid D_K}\left(1+\frac{1}{p}\right)^{-1}
	T(\log T)^{2},
\end{equation*}
where $L(s,\chi_{D_K})$ denotes the Dirichlet $L$-function associated with the primitive quadratic character $\chi_{D_K}(n)=\left(\frac{D_K}{n}\right)$, and $\zeta_K(s)=\zeta(s)L(s,\chi_{D_K})$. In support of the above prediction \eqref{conjecture-moment-dedekind}, for every finite Galois extension $K/\mathbb{Q}$, Akbary and Fodden \cite{Akbary-Fodden} established that for any rational number $k>0$ and for sufficiently large $T$,
	\begin{equation*}
			\int_{T}^{2T} \left| \zeta_K\!\left(\tfrac{1}{2}+it\right)\right|^{2k}\;dt 
		\gg_{K,k} T(\log T)^{[K:\mathbb{Q}]k^2}.
	\end{equation*}
	Subsequently, based on the method of Radziwi{\l}{\l} and Soundararajan \cite{Raziwill-Sound}, Sono \cite{K.Sono} extended this lower bound to all real $k>1$.

	On the other hand, using the method of Soundararajan \cite{Sound}, Milinovich and Turnage-Butterbaugh \cite{M-T} proved that, assuming GRH, if $K$ is a finite Galois extension of $\mathbb{Q}$, then for any fixed real $k>0$ and arbitrary $\varepsilon>0$,
	\begin{equation}\label{bound-milinovich-turnage}
		\int_{T}^{2T} \left| \zeta_K\!\left(\tfrac{1}{2}+it\right)\right|^{2k}\;dt 
		\ll_{K,k,\varepsilon} T(\log T)^{[K:\mathbb{Q}]k^2+\varepsilon},
	\end{equation}
where $T$ is sufficiently large.
Recently, Hagen \cite{Hagen} improved this bound for finite solvable Galois extensions $K/\mathbb{Q}$, proving that, for sufficiently large $T$,
\begin{equation*}
	\int_{T}^{2T} \left| \zeta_K\!\left(\tfrac{1}{2}+it\right)\right|^{2k}\,dt 
	\ll_{K,k} T(\log T)^{[K:\mathbb{Q}]k^2}.
\end{equation*}
The general shifted Dedekind-zeta problem for arbitrary fixed fields is not covered by those results. In this paper, we prove the conjecturally optimal upper bound in the broader case of all Galois and non-Galois extensions.

While the above results focus on moments of the Dedekind zeta function, it is natural to consider their shifted analogues.  Such moments have been widely studied for the Riemann zeta function \cite{Chandee_shifted,Curran-Correlation} and Dirichlet $L$-functions \cite{Munch,szabo}, but remain largely unexplored in the setting of Dedekind zeta functions. Very recently, Hsu and Wong \cite{Hsu-Wong} studied the log-correlations for shifted Dedekind zeta functions of quadratic fields. We prove the corresponding sharp shifted upper bounds for arbitrary fixed number fields.

	Let $r$ be a positive integer. Fix an algebraic closure $\overline{\mathbb Q}$ and embeddings $K_i\hookrightarrow\overline{\mathbb Q}$ for the fixed number fields $K_1,\ldots,K_r$.  Let $L\subset\overline{\mathbb Q}$ be the compositum of the corresponding Galois closures, regard every $K_i$ as a subfield of $L$, and set $G=\Gal(L/\mathbb Q)$ and $H_i=\Gal(L/K_i)$.  We use left cosets $G/H_i=\{gH_i:g\in G\}$ with the left action of $G$; the map $gH_i\mapsto g|_{K_i}$ identifies $G/H_i$ with the embeddings of $K_i$ in $L$.  Set
\begin{equation*}
\chi_i(\sigma)=|\Fix(\sigma:G/H_i\to G/H_i)|,\quad\sigma\in G.
\end{equation*}
For $1\le  i,j\le  r$, define
\begin{equation*}
        \calZ_{ij}(s)=
        \prod_{H_i gH_j\in H_i\backslash G/H_j}
        \zeta_{L^{H_i\cap gH_jg^{-1}}}(s).
\end{equation*}
This product is independent of the chosen representatives.  Indeed, replacing $g$ by $h_i g h_j$, with $h_i\in H_i$ and $h_j\in H_j$, replaces $H_i\cap gH_jg^{-1}$ by the conjugate subgroup $h_i(H_i\cap gH_jg^{-1})h_i^{-1}$.  So the fixed fields, $L^{H_i\cap gH_jg^{-1}}$ and $L^{h_i(H_i\cap gH_jg^{-1})h_i^{-1}}$ are $\mathbb Q$-isomorphic, and therefore have the same Dedekind zeta function.
For $T\ge3$, real shifts $\mathbf b=(b_1,\ldots,b_{r})$, and positive exponents $\mathbf a=(a_1,\ldots,a_{r})$, we set
\begin{equation*}
        B(T,\mathbf b,\mathbf a)=
        \prod_{i,j=1}^{r}
        \left|
        \calZ_{ij}\left(1+\frac1{\log T}+i(b_i-b_j)\right)
        \right|^{a_i a_j/4}.
\end{equation*}
We are now ready to state our main theorem.
	\begin{thm}\label{main-theorem}
Let $K_1,\ldots,K_r$ and $L$ be as above, and assume GRH for $\zeta_L(s)$. Let $\mathcal A\subset(0,\infty)^r$ be compact. There is $T_0=T_0(K_1,\ldots,K_r,\mathcal A)\ge3$ such that, for $T\ge T_0$ and $\mathbf a\in\mathcal A$, uniformly for real shifts satisfying $|b_i|\le T/2$,
		\begin{align*}
			\int_{T}^{2T} \prod_{j=1}^{r}\left| \zeta_{K_j}\left(\tfrac{1}{2}+i(t+b_j)\right)\right|^{a_j}\,dt
            \ll_{K_1,\ldots,K_{r},\mathcal A}
            T B(T,\mathbf b,\mathbf a).
		\end{align*}
		\end{thm}
The factor $B(T,\mathbf b,\mathbf a)$ records the singular correlations of the degree-one Frobenius coefficients. In the Galois specialisation $K_1=\cdots=K_r=K$, the permutation character is the regular character of $\Gal(K/\mathbb Q)$ and $\calZ_{ij}(s)=\zeta_K(s)^{[K:\mathbb Q]}$. For a non-Galois field the same covariance is expressed by the fixed fields of the double-coset stabilisers.

Our first corollary handles the unshifted case of Theorem \ref{main-theorem}.
		\begin{cor}\label{mixed-unshifted-corollary}
Let $K_1,\ldots,K_r$ be fixed number fields, and let $L$ be the compositum of their Galois closures over $\mathbb{Q}$. Let $G=\Gal(L/\mathbb{Q})$, $H_i=\Gal(L/K_i)$, and
\begin{equation*}
	\chi_i(\sigma)=|\Fix(\sigma:G/H_i\to G/H_i)|,\quad\sigma\in G.
\end{equation*}
Let $a_1,\ldots,a_r>0$ be fixed, and assume GRH for $\zeta_L(s)$. There is $T_0=T_0(K_1,\ldots,K_r,\mathbf a)\ge3$ such that the following holds for $T\ge T_0$:
			\begin{equation*}
				\int_T^{2T}\prod_{i=1}^{r}\left|\zeta_{K_i}\left(\tfrac{1}{2}+it\right)\right|^{a_i}\;dt
				\ll_{K_1,\ldots,K_r,\mathbf a}
				T(\log T)^{\frac{1}{4}\left\langle\sum_{i=1}^{r}a_i\chi_i,\sum_{i=1}^{r}a_i\chi_i\right\rangle_G},
			\end{equation*}
where
\begin{equation*}
	\langle\varphi,\psi\rangle_G
	=\frac{1}{|G|}\sum_{\sigma\in G}\varphi(\sigma)\overline{\psi(\sigma)}.
\end{equation*}
Equivalently, the exponent of $\log T$ is
\begin{equation*}
	\frac{1}{4}\sum_{i,j=1}^{r}a_ia_j|H_i\backslash G/H_j|.
\end{equation*}
		\end{cor}

\begin{proof}
Apply Theorem \ref{main-theorem} with $b_1=\cdots=b_r=0$. For every $1\le i,j\le r$, we have
\begin{equation*}
	\calZ_{ij}(s)=
	\prod_{H_igH_j\in H_i\backslash G/H_j}
	\zeta_{L^{H_i\cap gH_jg^{-1}}}(s).
\end{equation*}
Every Dedekind zeta function in this product has a simple pole at $s=1$ with positive residue. Therefore,
\begin{equation*}
	\calZ_{ij}\left(1+\frac{1}{\log T}\right)
	\asymp_{K_1,\ldots,K_r}
	(\log T)^{|H_i\backslash G/H_j|}.
\end{equation*}
Consequently,
\begin{equation*}
	B(T,\mathbf 0,\mathbf a)
	\asymp_{K_1,\ldots,K_r,\mathbf a}
	(\log T)^{\frac{1}{4}\sum_{i,j=1}^{r}a_ia_j|H_i\backslash G/H_j|}.
\end{equation*}
The diagonal action of $G$ on $G/H_i\times G/H_j$ has character $\chi_i\chi_j$, and its orbits are indexed by the double cosets $H_i\backslash G/H_j$. Hence Burnside's lemma gives
\begin{equation*}
	|H_i\backslash G/H_j|
	=\frac{1}{|G|}\sum_{\sigma\in G}\chi_i(\sigma)\chi_j(\sigma)
	=\langle\chi_i,\chi_j\rangle_G.
\end{equation*}
It follows that
\begin{equation*}
	\sum_{i,j=1}^{r}a_ia_j|H_i\backslash G/H_j|
	=\left\langle\sum_{i=1}^{r}a_i\chi_i,\sum_{i=1}^{r}a_i\chi_i\right\rangle_G.
\end{equation*}
Substituting this identity into the preceding estimate for $B(T,\mathbf 0,\mathbf a)$ and applying Theorem \ref{main-theorem} proves the corollary.
			\end{proof}

            The following corollary gives the upper bound for shifted moments in terms of the corresponding correlation functions of the Dedekind zeta function associated with a single number field.
            
\begin{cor}\label{galois-shifted-corollary}
Let $K$ be a fixed number field, let $L$ be its Galois closure over $\mathbb Q$, let $G=\Gal(L/\mathbb Q)$ and $H=\Gal(L/K)$, and set $\rho_K=|H\backslash G/H|$. Assume GRH for $\zeta_L(s)$, and let $a_1,\ldots,a_r>0$ be fixed. There is $T_0=T_0(K,r,\mathbf a)>10$ such that, for $T\ge T_0$, uniformly for real shifts satisfying $|b_i|\le T/2$,
\begin{equation*}
\int_T^{2T}\prod_{i=1}^{r}
\left|\zeta_K\left(\frac12+i(t+b_i)\right)\right|^{a_i}\,dt
\ll_{K,r,\mathbf a} T(\log T)^{\rho_K(a_1^2+\cdots+a_{r}^2)/4}
\prod_{1\le  i<j\le  r}
g(|b_i-b_j|)^{\rho_K a_i a_j/2},
\end{equation*}
where $g:[0,T]\to(0,\infty)$ is given by
\begin{equation*}
			g(x)= \begin{cases}
				\log T, \ & \text{if} \ 0\le x\le \frac{1}{\log T},\\
				x^{-1}, \ & \text{if} \ \frac{1}{\log T}<x\le 10,\\
				\log\log x, \ & \text{if}\ 10<x\le T.
			\end{cases}
\end{equation*}
\end{cor}
\begin{proof}
Apply Theorem \ref{main-theorem} with $K_i=K$ for every $i$. The common Galois closure is $L$, and $H_i=H$ for every $i$. Hence, for every ordered pair $(i,j)$, the double-coset product in Theorem \ref{main-theorem} is
\begin{equation*}
        \calZ_{ij}(s)=
        \prod_{HgH\in H\backslash G/H}
        \zeta_{L^{H\cap gHg^{-1}}}(s).
\end{equation*} Substitution in the definition of $B(T,\mathbf b,\mathbf a)$ gives
\begin{align*}
B(T,\mathbf b,\mathbf a)
&=
\prod_{i,j=1}^{r}
\prod_{HgH\in H\backslash G/H}
\left|
\zeta_{L^{H\cap gHg^{-1}}}
\left(
1+\frac1{\log T}+i(b_i-b_j)
\right)
\right|^{a_i a_j/4}.
\end{align*}
To obtain the required form with $g$, we separate the ordered pairs.  The diagonal contribution is
\begin{align*}
&\prod_{i=1}^{r}
\prod_{HgH\in H\backslash G/H}
\left|
\zeta_{L^{H\cap gHg^{-1}}}
\left(1+\frac1{\log T}\right)
\right|^{a_i^2/4}\\
&\qquad\asymp_{L,r,\mathbf a}
(\log T)^{\rho_K
(a_1^2+\cdots+a_r^2)/4}.
\end{align*}
For each $i<j$, the ordered pairs $(i,j)$ and $(j,i)$ combine to give
\begin{equation*}
\prod_{HgH\in H\backslash G/H}
\left|
\zeta_{L^{H\cap gHg^{-1}}}
\left(
1+\frac1{\log T}+i(b_i-b_j)
\right)
\right|^{a_i a_j/2}.
\end{equation*}
Since GRH holds for $\zeta_L(s)$, the Aramata--Brauer theorem \cite{Aramata, Brauer} implies that GRH also holds for $\zeta_{L^{H\cap gHg^{-1}}}(s)$. The remark following Lemma~B of Garaev, K\"uhleitner, Luca, and Nowak \cite[p.~304]{Garaev-Kuehleitner-Luca-Nowak} gives, under GRH, the fixed-field estimate $\zeta_F(1+iu)\ll_F\log\log |u|$ for sufficiently large $|u|$. Applying the Phragm\'en--Lindel\"of theorem in $1\le\Re s\le2$ to
\begin{equation*}
\frac{s-1}{s+1}\frac{\zeta_F(s)}{\log(\log(A+s))},
\end{equation*}
where $A>e^e+3$ is fixed and extends this estimate uniformly to $1\le\Re s\le2$. Consequently, uniformly for $|u|\le T$,
\begin{equation*}
        \left|\zeta_{L^{H\cap gHg^{-1}}}\left(1+\frac1{\log T}+iu\right)\right|\ll_L g(|u|).
\end{equation*}
For $|u|\le1/\log T$ the estimate follows from the simple pole of $\zeta_F(s)$ at $s=1$ and if $1/\log T<|u|\le10$ the Laurent expansion of $\zeta_F(s)$ gives
\begin{equation*}
    \zeta_F\!\left(1+\frac1{\log T}+iu\right)
\ll
\left|\frac1{1/\log T+iu}\right|+1
\ll
|u|^{-1}.
\end{equation*}

For $10<|u|\le T$ it follows from the preceding strip estimate with $\Re s=1+1/\log T$. These estimates are uniform over the finite set of intermediate fields in the product.
Applying this with $u=b_i-b_j$ gives the required result.
\end{proof}

     The next corollary establishes a sharp upper bound for the $2k$th moment of the Dedekind zeta function associated with an arbitrary number field.
            \begin{cor}\label{nongalois-unshifted-corollary}
Let $K$ be an arbitrary fixed number field, let $L$ be its Galois closure, let $G=\Gal(L/\mathbb Q)$ and $H=\Gal(L/K)$, and let $\chi(\sigma)=|\Fix(\sigma:G/H\to G/H)|$. If GRH holds for $\zeta_L(s)$ and $k>0$ is fixed, then there is $T_0=T_0(K,k)\ge3$ such that, for $T\ge T_0$,
\begin{equation*}
        \int_T^{2T}\left|\zeta_K\left(\frac12+it\right)\right|^{2k}\;dt
        \ll_{K,k}
        T(\log T)^{k^2\langle\chi,\chi\rangle_G},
\end{equation*}
where $\langle\chi,\chi\rangle_G=|G|^{-1}\sum_{\sigma\in G}\chi(\sigma)^2=|H\backslash G/H|$.
\end{cor}
\begin{proof}
Take $r=1$ and $a_1=2k$ in Corollary \ref{mixed-unshifted-corollary}.
\end{proof}
This is conjecturally sharp. Indeed, write
$\chi=\sum_{\pi\in\widehat G}m_\pi\chi_\pi$; then Artin formalism
gives $\zeta_K(s)=\prod_\pi L(s,\pi)^{m_\pi}$, and the standard
moment conjecture for non-primitive $L$-functions
\cite[Conjecture~4]{Heap1} predicts the logarithmic exponent
\begin{equation*}
k^2\sum_\pi m_\pi^2
=
k^2\langle\chi,\chi\rangle_G
=
k^2|H\backslash G/H|,
\end{equation*}
which is exactly the exponent obtained here.

The double cosets are the $H$-orbits on $G/H$. Hence
$\rho_K:=|H\backslash G/H|\le[G:H]=[K:\mathbb Q]$, with equality precisely when every $H$-orbit is a singleton. This is equivalent to $g^{-1}Hg=H$ for every $g\in G$, and therefore to normality of $H$. Since $L$ is the Galois closure of $K$, equality holds exactly when $K/\mathbb Q$ is Galois.

As an illustration, let $K$ be a non-Galois cubic extension of $\mathbb{Q}$ with discriminant $D$, where $D$ is not a perfect square. Then its Galois closure is $L=K(\sqrt{D})$, and the Galois group $G=\Gal(L/\mathbb{Q})$ is isomorphic to $S_3$. Moreover,
$H=\Gal(L/K)\cong C_2$. The permutation character on the three left cosets has values $3$, $1$, and $0$ on the identity, the three transpositions, and the two $3$-cycles. Burnside's lemma therefore gives
\begin{equation*}
|H\backslash G/H|=\frac{3^2+3\cdot1^2+2\cdot0^2}{6}=2.
\end{equation*}
Hence, by the preceding corollary, for every fixed $k>0$,
\begin{equation*}
\int_T^{2T}
\left|
\zeta_K\left(\frac12+it\right)
\right|^{2k}\,dt
\ll_{K,k}
T(\log T)^{2k^2}.
\end{equation*}
Here $[K:\mathbb{Q}]=3$, whereas the logarithmic exponent is determined by $|H\backslash G/H|=2$. In general the preceding orbit argument gives strict inequality for every non-Galois field.

The following corollary improves the conditional upper bound of Milinovich and Turnage-Butterbaugh \cite{M-T}. Moreover, it extends Hagen's \cite{Hagen} result, which was established for finite solvable Galois extensions, to any finite Galois extension $K/\mathbb{Q}$. This is an immediate consequence of Corollary~\ref{nongalois-unshifted-corollary} upon taking $L=K$ and $H=\{1\}$, since $K/\mathbb{Q}$ is Galois.

		\begin{cor}\label{sharp-bound-dedekind-zeta}
Let $K/\mathbb Q$ be a finite Galois extension and let $k>0$. Assume GRH for $\zeta_K(s)$. There is $T_0=T_0(K,k)\ge3$ such that, for $T\ge T_0$,
			\begin{equation*}
				\int_{T}^{2T}\left|\zeta_K\left(\tfrac{1}{2}+it\right) \right|^{2k}\;dt\ll_{K,k} T(\log T)^{[K:\mathbb{Q}]k^2}.
			\end{equation*}
		\end{cor}

The bounds in Corollaries \ref{nongalois-unshifted-corollary} and \ref{sharp-bound-dedekind-zeta} are uniform for $k$ in a fixed compact interval $[k_1,k_2]\subset(0,\infty)$. Indeed, apply the compact-uniform form of Theorem \ref{main-theorem} with the exponent $2k$ ranging over $[2k_1,2k_2]$.

\subsection{Conditional upper bounds for large deviations of $\zeta_K(1/2+it)$}
The study of large deviations of $\log |\zeta(1/2+it)|$ is closely connected with moments of the Riemann zeta function through the identity
\begin{equation*}
\int_T^{2T}|\zeta(\tfrac{1}{2}+it)|^{2k}\;dt=\int_T^{2T}e^{2k\log|\zeta(\tfrac{1}{2}+it)|}\;dt.
\end{equation*}
Under the Riemann Hypothesis, a series of works has shown that the large deviation regime $V\asymp \log\log T$ has Gaussian-type decay and determines the size of high moments. Harper's bound \eqref{harper-upper-bound} directly gives that for $\log\log T \leq V\leq 1000\log\log T$, 
	\begin{equation*}
		\text{meas}\{ t\in[T,2T]: \log|\zeta(1/2+it)|\geq V\}\ll Te^{-V^2/\log\log T}.
	\end{equation*}
	However, the Gaussian heuristic suggests that the true order should be smaller by a factor $(\log\log T)^{-1/2}$. Under RH, Arguin, Bailey, and Roberts \cite{Arguin-Bailey-Roberts} proved that, for each fixed $k>0$ and all sufficiently large $T$,
\begin{equation*}
\frac1T\operatorname{meas}\left\{t\in[T,2T]:
\left|\zeta\left(\frac12+it\right)\right|>(\log T)^k\right\}
\le \frac{\exp(e^{ck})}{\sqrt{\log\log T}}(\log T)^{-k^2},
\end{equation*}
where $c>0$ is absolute. Their upper bound has the order predicted by the Gaussian heuristic. Arguin and Bailey \cite{Arguin-Bailey} proved the corresponding upper bound unconditionally for fixed $0<k<2$. These results motivate the study of analogous large deviation bounds for the Dedekind zeta function.
	
	    As an application, we study the distribution of large values of the Dedekind zeta function on the critical line. The following corollary gives an upper bound on the measure of level sets of $\zeta_K(1/2+it)$ in the interval $[T,2T]$ for values of the order $\log\log T$, conditionally on GRH.

\begin{cor}\label{large-deviation-nongalois}
Let $K$ be a fixed number field, let $L$ be its Galois closure over $\mathbb Q$, let $G=\Gal(L/\mathbb Q)$ and $H=\Gal(L/K)$, and set
\begin{equation*}
        \rho_K=|H\backslash G/H|.
\end{equation*}
Assume GRH for $\zeta_L(s)$. Fix constants $0<c_1<c_2$. There is $T_0=T_0(K,c_1,c_2)\ge3$ such that, for $T\ge T_0$, uniformly for
\begin{equation*}
        c_1\sqrt{\log\log T}\le  V\le  c_2\sqrt{\log\log T},
\end{equation*}
one has
\begin{equation*}
        \frac1T\operatorname{meas}\left\{
        t\in[T,2T]:
        \log\left|\zeta_K\left(\frac12+it\right)\right|
        \ge
        V\sqrt{\frac{\rho_K}{2}\log\log T}
        \right\}
        \ll_{K,c_1,c_2}
        e^{-V^2/2}.
\end{equation*}
If $K/\mathbb Q$ is Galois, then $\rho_K=[K:\mathbb Q]$.
\end{cor}
\begin{proof}
Let
\begin{equation*}
        k=\frac{V}{\sqrt{2\rho_K\log\log T}}.
\end{equation*}
The condition $V\asymp\sqrt{\log\log T}$ keeps $k$ in a fixed compact subinterval of $(0,\infty)$.  Therefore, Corollary \ref{nongalois-unshifted-corollary} and Chebyshev's inequality give
\begin{align*}
&\operatorname{meas}\left\{
        t\in[T,2T]:
        \log\left|\zeta_K\left(\frac12+it\right)\right|
        \ge
        V\sqrt{\frac{\rho_K}{2}\log\log T}
        \right\} \\
&\qquad\le
        \exp\left(-2k V\sqrt{\frac{\rho_K}{2}\log\log T}\right)
        \int_T^{2T}
        \left|\zeta_K\left(\frac12+it\right)\right|^{2k}\,dt  \\
&\qquad\ll
        T\exp\left(-V^2+\rho_Kk^2\log\log T\right)
        =
        T e^{-V^2/2}.
\end{align*}
This proves the corollary. 
\end{proof}

\subsection{Coefficients of Dedekind zeta functions in short intervals}
Let $K$ be a number field and let $\mathcal O_K$ be its ring of integers. For a positive integer $n$, let $r_K(n)$ denote the number of integral ideals of norm $n$. Then, for $\Re s>1$,
\begin{equation}\label{Dedekind-zeta-dirichlet-series}
        \zeta_K(s)=\sum_{n=1}^{\infty}\frac{r_K(n)}{n^s}.
\end{equation}
The residue at $s=1$ is
\begin{equation}\label{eq:class-number-formula}
        \operatorname*{Res}_{s=1}\zeta_K(s)
        =
        \frac{2^{r_1}(2\pi)^{r_2}h_KR_K}
        {w_K\sqrt{|D_K|}},
\end{equation}
where $r_1$ and $r_2$ denote the numbers of real embeddings and complex embeddings of $K$ up to conjugation, respectively, $h_K$ is the class number of $K$, $R_K$ is the regulator of $K$, $w_K$ is the number of roots of unity in $K$, and $D_K$ is the discriminant of $K$. Landau's classical estimate \cite{Landau} gives
\begin{equation*}
        \sum_{n\le  x}r_K(n)
        =
        \frac{2^{r_1}(2\pi)^{r_2}h_KR_K}
        {w_K\sqrt{|D_K|}}x
        +O_K\left(x^{1-2/([K:\mathbb Q]+1)}\right).
\end{equation*}
Adapting the Perron--Plancherel argument used in the proof of Milinovich and Turnage-Butterbaugh \cite[Theorem~1.4]{M-T}, we obtain a proportional-interval estimate. The proof is included because the cited theorem assumes a relative interval length tending to zero.

\begin{cor}\label{short-interval-nongalois}
Let $K$ be a fixed number field, let $L$ be its Galois closure over $\mathbb Q$, let $G=\Gal(L/\mathbb Q)$ and $H=\Gal(L/K)$, and set $\rho_K=|H\backslash G/H|$. Assume GRH for $\zeta_L(s)$. There is $X_0=X_0(K)\ge3$ such that, for $X\ge X_0$, uniformly for $2\le Q\le X$,
\begin{equation*}
        \frac1X\int_X^{2X}
        \left|
        \sum_{x<n\le x+x/Q}r_K(n)
        -
        \frac{2^{r_1}(2\pi)^{r_2}h_KR_K}
        {w_K\sqrt{|D_K|}}\frac{x}{Q}
        \right|^2\,dx
        \ll_K
        \frac{X}{Q}\left(\log(2Q)\right)^{\rho_K}.
\end{equation*}
Moreover, for every $\psi(X)\ge1$, the estimate
\begin{equation*}
        \sum_{x<n\le x+x/Q}r_K(n)
        =
        \frac{2^{r_1}(2\pi)^{r_2}h_KR_K}
        {w_K\sqrt{|D_K|}}\frac{x}{Q}
        +
        O_K\left(
        \sqrt{\frac{X}{Q}}\left(\log(2Q)\right)^{\rho_K/2}\psi(X)
        \right)
\end{equation*}
holds for all $x\in[X,2X]$ outside a set of measure $O_K(X/\psi(X)^2)$.
\end{cor}

\begin{proof}
Applying Corollary~\ref{nongalois-unshifted-corollary} with $k=1$, together with conjugation and a dyadic decomposition, we obtain
\begin{equation*}
\int_{-U}^{U}\left|\zeta_K\left(\frac12+it\right)\right|^2\,dt
\ll_K U\left(\log(2U)\right)^{\rho_K}.
\end{equation*}
If $s=1/2+it$, then
\begin{equation*}
\left|\frac{(1+1/Q)^s-1}{s}\right|
\ll\min\left\{\frac1Q,\frac1{1+|t|}\right\}.
\end{equation*}
The first estimate follows from
\begin{equation*}
    \frac{(1+1/Q)^s-1}{s}
=\int_0^{\log(1+1/Q)}e^{us}\,du \ll Q^{-1},
\end{equation*}
as $\Re s=1/2$. For the second, note that
$|(1+1/Q)^s-1|\ll1$,
so
\begin{equation*}
    \left|\frac{(1+1/Q)^s-1}{s}\right|
\ll |s|^{-1}
\asymp (1+|t|)^{-1}.
\end{equation*}

 Combining this with the second-moment estimate above and summing over dyadic ranges, we obtain
\begin{align}
\nonumber &\int_{-\infty}^{\infty}
\left|\zeta_K\left(\frac12+it\right)\right|^2
\left|\frac{(1+1/Q)^{1/2+it}-1}{1/2+it}\right|^2\,dt\\
\nonumber&\quad\ll_K Q^{-2}\int_{|t|\le Q}
\left|\zeta_K\left(\frac12+it\right)\right|^2\,dt\\
\nonumber&\qquad+\sum_{j\ge0}(2^jQ)^{-2}
\int_{2^jQ<|t|\le2^{j+1}Q}
\left|\zeta_K\left(\frac12+it\right)\right|^2\,dt\\
&\quad\ll_K\frac1Q\left(\log(2Q)\right)^{\rho_K}. \label{eq: estimate-second-moment}
\end{align}
We next justify the passage from the Dirichlet series to the short-interval sum via Perron's formula and Plancherel's theorem. By Perron's formula, 
\begin{equation*}
\frac1{2\pi i}\int_{2-i\infty}^{2+i\infty}
\zeta_K(s)\frac{(1+1/Q)^s-1}{s}x^s\,ds =\sum_{x<n\le x+x/Q}r_K(n),
\end{equation*}
for every $x>0$ except $x$ and $x+x/Q$ are not integers.
By the Aramata--Brauer theorem, GRH for $\zeta_L$ implies GRH for $\zeta_K$. Consequently, the generalized Lindelöf hypothesis holds for $\zeta_K$, so that for every $\varepsilon>0$ and $1/2\le\sigma\le2$,
\begin{equation*}
\zeta_K(\sigma\mathbin{\pm}iR)\ll_{K,\varepsilon}R^\varepsilon.
\end{equation*}
We now shift the contour from the line segment $[2-iR,\,2+iR]$ to $[1/2-iR,\,1/2+iR]$. The only pole crossed is at $s=1$, with residue
\begin{equation*}
\frac{2^{r_1}(2\pi)^{r_2}h_KR_K}
{w_K\sqrt{|D_K|}}\frac{x}{Q}.
\end{equation*}
On the horizontal segments,
$((1+1/Q)^s-1)/s\ll R^{-1} $ ,
and hence their total contribution is
$O(x^2R^{\varepsilon-1})$,
which tends to zero as $R\to\infty$. Since the critical-line integral is square-integrable, we may pass to the limit in $L^2$. After the change of variables $x=e^\tau$, Plancherel's theorem yields
\begin{align*}
&\int_0^\infty
\left|\sum_{x<n\le x+x/Q}r_K(n)
-\frac{2^{r_1}(2\pi)^{r_2}h_KR_K}
{w_K\sqrt{|D_K|}}\frac{x}{Q}\right|^2\frac{dx}{x^2}\\
&\qquad=\frac1{2\pi}\int_{-\infty}^{\infty}
\left|\zeta_K\left(\frac12+it\right)\right|^2
\left|\frac{(1+1/Q)^{1/2+it}-1}{1/2+it}\right|^2\,dt.
\end{align*}
Restricting the left-hand side to $X\le x\le2X$, using $x^{-2}\ge (2X)^{-2}$, and applying the estimate \eqref{eq: estimate-second-moment}, we obtain the desired mean-square bound. Finally, Chebyshev's inequality, applied with threshold
\begin{equation*}
\sqrt{\frac{X}{Q}}\left(\log(2Q)\right)^{\rho_K/2}\psi(X),
\end{equation*}
yields the exceptional-set estimate.
\end{proof}
\subsection*{Overview of the proof}
Our proof extends Harper's method for moments of the Riemann zeta function to products of shifted Dedekind zeta functions attached to arbitrary number fields. The principal new difficulty occurs in the non-Galois setting, where the covariance of the prime coefficients is naturally expressed in terms of permutation characters of the Galois closure and the corresponding Artin L-functions. This issue arises in the analysis of the prime-square contribution in Chandee's explicit formula. To overcome it, we use Artin induction together with the Aramata--Brauer theorem to rewrite the relevant Euler products in terms of Dedekind zeta functions of intermediate fields, thereby avoiding any appeal to the Artin conjectures.

The first step is to establish a Chandee-type upper bound for
\begin{equation*}
    \sum_{j=1}^{r}a_j
\log\left|
\zeta_{K_j}\!\left(\frac12+i(t+b_j)\right)
\right|.
\end{equation*}
Starting from Chandee's explicit formula, we majorise each logarithm by a short Dirichlet polynomial over primes and prime squares. The contribution of higher prime powers is negligible, while the prime-square terms are treated using the above reduction to Dedekind zeta functions. This yields Proposition~\ref{main-prop}, which furnishes a uniform upper bound for products of shifted Dedekind zeta functions.

With Proposition~\ref{main-prop} in hand, we adapt Harper's argument. The primes are partitioned into short intervals, and the interval $[T,2T]$ is decomposed into a good set and exceptional sets according to the size of the corresponding Dirichlet polynomials. On the good set, the exponential is replaced by a truncated exponential series, and the resulting mean value is estimated using orthogonality of Dirichlet polynomials. Proposition~\ref{frobenius-orbit-product} identifies the resulting diagonal contribution with the Euler product
$B(T,\mathbf b,\mathbf a)$,
while the off-diagonal terms are negligible because the Dirichlet polynomials are sufficiently short.

Finally, the exceptional sets are handled using Harper's iterative argument with progressively shorter Dirichlet polynomials, yielding an exponential saving. Combining the estimates for the good and exceptional sets establishes Theorem~\ref{main-theorem}. 
\section*{Acknowledgements}
		
		The authors are grateful to Markus V.~Hagen, Winston Heap, and Micah B.~Milinovich for valuable discussions related to this work. BD is supported by the Heilbronn Institute for Mathematical Research. NK is supported by the Prime Minister's Research Fellowship (PMRF), Government of India (PMRF ID:~2403449). KM is supported by the ARG-MATRICS program (Grant No.~ANRF/ARGM/2025/002540/MTR).
		
	\section{Preliminary tools}
The following auxiliary results will be used repeatedly in the proofs of the main theorems.
	\begin{lem}\label{truncate-dirichlet-poly}
			Let $a(n)$ be a sequence of complex numbers and define $D(s)= \sum_{p\le  X}a(p)p^{-s}$ with $s=\sigma +it $. Assume that $V$  is larger than a sufficiently large absolute constant. Suppose $t\in [T,2T]$ is such that $|D(s)|\le  V$, and let $N\ge 10V$ be an integer.  Then
			\begin{equation*}
				\exp\left(2\Re D(s)\right) = \left(1+O\left(e^{-9V}\right)\right)\left| \sum_{j\le  N} \frac{(D(s))^j}{j!}\right|^2.
			\end{equation*}
		\end{lem}
	\begin{proof}
		Let $z=D(s)$ and
        \begin{equation*}
                P_N(z)=\sum_{0\le j\le N}\frac{z^j}{j!}.
        \end{equation*}
        Since $N\ge10V$, successive terms in the exponential tail have ratio at most $V/(N+2)\le1/10$.  Hence
        \begin{align*}
        \left|e^z-P_N(z)\right|
        &\le \sum_{j>N}\frac{V^j}{j!}
        \le \frac{10}{9}\frac{V^{N+1}}{(N+1)!}.
        \end{align*}
        Using $(N+1)!\ge((N+1)/e)^{N+1}$ and $-\Re z\le V$, we obtain
        \begin{equation*}
        \left|e^{-z}P_N(z)-1\right|
        \le
        \frac{10}{9}e^V
        \left(\frac{eV}{N+1}\right)^{N+1}.
        \end{equation*}
Since $N+1\ge 10V$, we obtain 
$\left|e^{-z}P_N(z)-1\right|
\ll
\exp\left(-\left(10\log(10/e)-1\right)V\right)
\ll e^{-9V}.$
Therefore, 
\begin{equation*}
    P_N(z)=e^z\left( 1+ O\left(e^{-9V}\right)\right).
\end{equation*}
Taking absolute values squared yields
        \begin{equation*}
                e^{2\Re z}
=\left(1+O(e^{-9V})\right)|P_N(z)|^2.
\end{equation*}
        All implied constants are absolute.
	\end{proof}

\begin{lem}\label{complex-dirichlet-mean}
Let $\mathcal P_1,\ldots,\mathcal P_R$ be disjoint finite sets of primes and
let $\mathcal Q_1,\ldots,\mathcal Q_S$ be finite sets whose elements are
pairwise coprime, both within and between the sets, and coprime to every prime
in the $\mathcal P_r$.  Let
\begin{equation*}
D_r(t)=\sum_{p\in\mathcal P_r}c_r(p)p^{-it},\qquad
E_s(t)=\sum_{q\in\mathcal Q_s}d_s(q)q^{-it}.
\end{equation*}
Let $N_r,M_s$ be non-negative integers.  Suppose every integer occurring in
\begin{equation*}
 A(t):=\prod_{s=1}^S E_s(t)^{M_s}
 \prod_{r=1}^R\sum_{0\le m\le N_r}\frac{D_r(t)^m}{m!}
\end{equation*}
is at most $T^{1/10}$.  Then    
        \begin{align*}
         &\int_T^{2T}
        \left|\prod_{r=1}^{R}\sum_{0\le  m\le  N_r}\frac{D_r(t)^m}{m!}\right|^2
        \left|\prod_{s=1}^S E_s(t)^{M_s} \right|^2\,dt \\
&\ll T\exp\left(\sum_{r=1}^R\sum_{p\in\mathcal P_r}|c_r(p)|^2\right)
 \prod_{s=1}^S M_s!\left(\sum_{q\in\mathcal Q_s}|d_s(q)|^2\right)^{M_s}\\
&\quad+T^{1/10}
 \prod_{s=1}^S\left(\sum_{q\in\mathcal Q_s}|d_s(q)|\right)^{2M_s}
 \left(\prod_{r=1}^R\sum_{0\le m\le N_r}\frac1{m!}
 \left(\sum_{p\in\mathcal P_r}|c_r(p)|\right)^m\right)^2.
\end{align*}
Here empty products are interpreted as $1$.
\end{lem}
\begin{proof}
Write $A(t)=\sum_na(n)n^{-it}$. 
The diagonal contribution is $T\sum_n|a(n)|^2$.  If $m,n\le T^{1/10}$ and
$m\ne n$, then $|\log(n/m)|\gg T^{-1/10}$, so the off-diagonal contributions are
bounded by $T^{1/10}(\sum_n|a(n)|)^2$.

By unique factorisation, the contributions from the sets $\mathcal P_r$ and $\mathcal Q_s$ separate multiplicatively. For each auxiliary set, the multinomial theorem and
$\prod_q\nu(q)!\ge1$ give
\begin{equation*}
 \sum_{\sum_q\nu(q)=M_s}
 \left(\frac{M_s!}{\prod_q\nu(q)!}\right)^2
 \prod_q|d_s(q)|^{2\nu(q)}
 \le M_s!\left(\sum_q|d_s(q)|^2\right)^{M_s}.
\end{equation*}

The truncated exponential factors contribute at most
$\exp(\sum_{r,p}|c_r(p)|^2)$ to $\sum_n|a(n)|^2$, while
\begin{equation*}
 \sum_n|a(n)|\le
 \prod_s\left(\sum_q|d_s(q)|\right)^{M_s}
 \prod_r\sum_{m\le N_r}\frac1{m!}
 \left(\sum_p|c_r(p)|\right)^m.
\end{equation*}
 This proves the lemma.
\end{proof}

	\subsection{Conditional approximations of Dedekind zeta functions} Now we will obtain a finite Dirichlet series approximation for $\log|\zeta_K(1/2+it)|$. Similar approximations were used by Soundararajan \cite{Sound} and Harper \cite{Harper} to obtain moment bounds. The Generalised Riemann Hypothesis is used at this step, and we assume it throughout this subsection. In the non-Galois setting, we use Artin induction together with the Aramata--Brauer theorem to extend these arguments to Dedekind zeta functions.
    
	Taking logarithmic derivatives in \eqref{dedekind-zeta},
	for $\Re(s)>1$, we have
	\begin{equation*}
		\frac{\zeta_K'}{\zeta_K}(s)= -\sum_{n=1}^{\infty} \frac{\Lambda_K(n)}{n^s}.
	\end{equation*}
		From the Euler product one has the elementary coefficient bound $|\Lambda_K(n)|\le  [K:\mathbb{Q}]\Lambda(n)$.
\begin{lem}\label{conditional-bound-chandee}
Let $K$ be a fixed number field, let $d=[K:\mathbb Q]$, and assume GRH for $\zeta_K(s)$.  Define
\begin{equation*}
        \mathcal C_K(t)=|D_K|(|t|+3)^d.
\end{equation*}
There is $T_0=T_0(K)\ge3$ such that, whenever $T\ge T_0$, $T/2\le t\le5T/2$, and $e^2\le x\le T^2$, one has
\begin{align*}
\log\left|\zeta_K\left(\frac12+it\right)\right|
&\le
\Re\sum_{n\le x}
\frac{\Lambda_K(n)}
{n^{1/2+1/\log x+it}\log n}
\frac{\log(x/n)}{\log x} \\
&\quad+
\frac{\log\mathcal C_K(t)}{\log x}
+O_K\left(\frac1{\log x}\right),
\end{align*}
\end{lem}
\begin{proof}
Let $(r_1,r_2)$ be the signature of $K$, so that $d=r_1+2r_2$.  The logarithm of Chandee's analytic conductor for the shifted Dedekind zeta function is
\begin{equation*}
\log\left(\frac{|D_K|}{\pi^d}
\left|\frac14+\frac{it}{2}\right|^{r_1+r_2}
\left|\frac34+\frac{it}{2}\right|^{r_2}\right)
=\log\mathcal C_K(t)+O_K(1).
\end{equation*}
We apply Chandee's argument \cite[Theorem~2.1 and Section~3]{Chandee} with $\lambda=1$; the condition $\lambda\le\log x/2$ is then $x\ge e^2$.  The two additional terms in Theorem~2.1 are
\begin{equation*}
\frac{2d}{\log^2x}+\frac{4de^{-1}}{x^{1/2}\log^2x},
\end{equation*}
and these are $O_K(1/\log x)$.  For completeness, the smoothed residue terms arising at $s=1$ and $s=0$, with $z=\sigma+it$, are
\begin{equation*}
\frac{1-x^{1-z}}{(z-1)^2\log x}-\frac1{z-1},
\qquad -\frac{x^{-z}}{z^2\log x}.
\end{equation*}
Uniformly for $1/2\le\sigma\le1/2+1/\log x$ in the stated ranges, their absolute values are respectively at most
\begin{equation*}
\frac1{|t|}+\frac{1+x^{1/2}}{t^2\log x},
\qquad \frac{x^{-1/2}}{(t^2+1/4)\log x},
\end{equation*}
so their sum is $O_K(1/\log x)$.  Replacing Chandee's conductor by $\mathcal C_K(t)$ changes the right-hand side by the same amount.  This proves the lemma.
\end{proof}
\begin{lem}\label{artin-square-coefficients}
Let $L/\mathbb Q$ be finite Galois, let $G=\Gal(L/\mathbb Q)$, and let $K=L^H$.  If $\chi_H=\operatorname{Ind}_H^G\mathbf 1_H$, then there are cyclic subgroups $C_1,\ldots,C_m\le G$ and rational numbers $q_1,\ldots,q_m$ such that
\begin{equation}\label{eq:artin-square-decomposition}
\chi_H(\sigma^2)
=\sum_{\nu=1}^m q_\nu
\operatorname{Ind}_{C_\nu}^G\mathbf 1_{C_\nu}(\sigma).
\end{equation}
Consequently, for every rational prime $p$ unramified in $L$,
\begin{equation}\label{eq:artin-square-prime}
\frac{\Lambda_K(p^2)}{\log p}
=\sum_{\nu=1}^m q_\nu
\frac{\Lambda_{L^{C_\nu}}(p)}{\log p}.
\end{equation}
\end{lem}
\begin{proof}
Let $X = G/H$, so that $\chi_H(\sigma)$ is the number of fixed
points of $\sigma$ on $X$.  Let $f(\sigma)$ be the number of fixed
points of $\sigma$, and let $t(\sigma)$ be the number of two-cycles
in the cycle decomposition of $\sigma$ on $X$.

Let $\mathcal P_2(X)$ be the set of two-element subsets of $X$, and
denote its permutation character by $\chi_{\mathcal P_2(X)}$.  A
two-element subset of $X$ is fixed by $\sigma$ precisely when its two
elements are both fixed by $\sigma$ or when they form a two-cycle.
Consequently,
\begin{equation*}
\chi_{\mathcal P_2(X)}(\sigma)
=\binom{f(\sigma)}{2}+t(\sigma).
\end{equation*}
The permutation character of the diagonal action of $G$ on $X\times X$
has value $f(\sigma)^2$.  Since $\sigma^2$ fixes exactly the points
belonging to cycles of $\sigma$ of length one or two, it follows that
\begin{equation*}
\begin{split}
\chi_H(\sigma^2)
&=f(\sigma)+2t(\sigma)\\
&=2\chi_{\mathcal P_2(X)}(\sigma)
  +2\chi_H(\sigma)-\chi_{X\times X}(\sigma).
\end{split}
\end{equation*}
Thus the class function $\sigma\mapsto\chi_H(\sigma^2)$ is an integral
linear combination of permutation characters defined over $\mathbb Q$.
By rational Artin induction (see \cite[Chapter 13, Section 1, Theorem 30]{Serre}), the class function is a finite rational linear combination of permutation characters induced from cyclic subgroups.  Hence there are cyclic subgroups
$C_1,\ldots,C_m\le G$ and rational numbers $q_1,\ldots,q_m$ for which
\eqref{eq:artin-square-decomposition} holds.

Now let $p$ be unramified in $L$.  More generally, let $J\le G$ and
write $F = L^J$.  The cycles of $\Frob_p$ on $G/J$ correspond to the
primes of $F$ above $p$, and the length of each cycle is the residue
degree of the corresponding prime.  If $r_f(F,p)$ denotes the number
of primes of $F$ above $p$ having residue degree $f$, then
\begin{equation*}
\operatorname{Ind}_J^G\mathbf 1_J(\Frob_p^a)
=\sum_{f\mid a} f\,r_f(F,p).
\end{equation*}
On the other hand, the Euler factor of $\zeta_F(s)$ at $p$ gives
\begin{equation*}
\frac{\Lambda_F(p^a)}{\log p}
=\sum_{f\mid a} f\,r_f(F,p).
\end{equation*}
Therefore
\begin{equation*}
\frac{\Lambda_{L^J}(p^a)}{\log p}
=\operatorname{Ind}_J^G\mathbf 1_J(\Frob_p^a).
\end{equation*}
Taking $J=H$ and $a=2$, and then applying
\eqref{eq:artin-square-decomposition}, gives
\begin{equation*}
\begin{split}
\frac{\Lambda_K(p^2)}{\log p}
&=\chi_H(\Frob_p^2)\\
&=\sum_{\nu=1}^m q_\nu
\operatorname{Ind}_{C_\nu}^G\mathbf 1_{C_\nu}(\Frob_p)\\
&=\sum_{\nu=1}^m q_\nu
\frac{\Lambda_{L^{C_\nu}}(p)}{\log p},
\end{split}
\end{equation*}
which proves \eqref{eq:artin-square-prime}.
\end{proof}
The following lemma establishes an explicit formula for the twisted Chebyshev function associated with $F$.
\begin{lem}\label{lem:twisted-chebyshev}
Let $F$ be a fixed number field, assume GRH for $\zeta_F(s)$, and let
$T\ge3$.  For
\begin{equation*}
\Psi_F(z,\tau)=\sum_{n\le z}\Lambda_F(n)n^{-i\tau},
\end{equation*}
one has
\begin{equation}\label{eq:twisted-chebyshev-bound}
\Psi_F(z,\tau)\ll_F\frac zT+z^{1/2}\log^2T,
\end{equation}
uniformly for $2\le z\le T$ and $T\le\tau\le5T$.
\end{lem}
\begin{proof}
Set $c=1+1/\log z$.  The zero-counting formula for the completed
Dedekind zeta function gives, with $\rho_F=\beta_F+i\gamma_F$ running over
the nontrivial zeros counted with multiplicity,
\begin{equation}\label{eq:unit-zero-count}
\#\{\rho_F=\beta_F+i\gamma_F:u\le\gamma_F\le u+1\}
\ll_F\log(|u|+3).
\end{equation}
The partial-fraction expansion gives, uniformly for
$-1/2\le\sigma\le c$, $|v|\ge2$, and away from zeros,
\begin{equation}\label{eq:log-derivative-local}
\frac{\zeta_F'}{\zeta_F}(\sigma+iv)
=\sum_{|\gamma_F-v|\le1}
\frac1{\sigma+iv-\rho_F}+O_F(\log(|v|+3)).
\end{equation} 
It follows from \eqref{eq:unit-zero-count} that one may choose
$U\in[6T,7T]$ such that both $\tau+U$ and $\tau-U$ have distance
$\gg_F1/\log T$ from every ordinate $\gamma_F$.  Such $U$ exists since the intervals
of length a sufficiently small fixed multiple of $1/\log T$ excluded by
the $O_F(T\log T)$ ordinates meeting the relevant range have total length
less than $T/2$. Neither horizontal side passes through the pole or a trivial zero, since
$\tau+U\ge7T$ and $\tau-U\le-T$.

Perron's formula and
$|\Lambda_F(n)|\le[F:\mathbb Q]\Lambda(n)$ give
\begin{align}
\Psi_F(z,\tau)
&=\frac1{2\pi i}\int_{c-iU}^{c+iU}
-\frac{\zeta_F'}{\zeta_F}(w+i\tau)
\frac{z^w}{w}\,dw
+O_F\left(\log(2z)+\frac{z\log^2T}{U}\right).
\label{eq:twisted-perron}
\end{align}

To verify the error, it is enough to bound
\begin{equation}\label{eq: error_in_perron}
\sum_{n=1}^{\infty}\Lambda_F(n)\left(\frac zn\right)^c
\min\left\{1,\frac1{U|\log(z/n)|}\right\}.
\end{equation}
For the at most three integers with $|n-z|\le1$ the summands contribute $O_F(\log(2z))$. Now assume $|n-z|>1$. For $z/2<n<2z$, we have
$|\log(z/n)|\asymp|n-z|/z$, so \eqref{eq: error_in_perron} is
\begin{equation*}
\ll_F\frac{z^{c}\log(2z)}U\sum_{1\le m\le2z}\frac1m
\ll_F\frac{z\log(2z)\log(2U)}U.
\end{equation*}
If $n\le z/2$ and $n\ge2z$, then $|\log(z/n)|\gg 1$, so \eqref{eq: error_in_perron} is bounded by 
\begin{equation*}
    \frac{z}{U}\frac{\zeta_F'}{\zeta_F}(c) \ll \frac{z\log z}{U}.
\end{equation*}
 This proves the error term
in \eqref{eq:twisted-perron}.  At an integral endpoint, Perron's formula
has half weight and restoring the full endpoint costs $O_F(\log(2z))$.

Move the contour in \eqref{eq:twisted-perron} to $\Re w=-1/2$.  The contour crosses the singularities
\begin{equation*}
w=1-i\tau,\qquad w=\rho_F-i\tau,\qquad w=-i\tau,\qquad w=0.
\end{equation*}
The point $w=-i\tau$ contributes only when $\zeta_F$ has a zero at
$s=0$; its order is $r_1(F)+r_2(F)-1$.  By contrast, $w=0$ always
arises from the Perron factor $1/w$.  All negative trivial zeros lie to
the left of the new contour.  The residue theorem therefore gives
\begin{align}
\Psi_F(z,\tau)
&=\frac{z^{1-i\tau}}{1-i\tau}
-\sum_{\substack{\rho_F=1/2+i\gamma_F\\
|\gamma_F-\tau|<U}}
\frac{z^{\rho_F-i\tau}}{\rho_F-i\tau}
-\frac{\zeta_F'}{\zeta_F}(i\tau) \nonumber\\
&\quad+
\frac{(r_1(F)+r_2(F)-1)z^{-i\tau}}{i\tau}
+O_F\left(\log^2T
+z^{-1/2}\log^2T\right).
\label{eq:twisted-explicit-formula}
\end{align}
For $-1/2\le\sigma\le c$, equation \eqref{eq:log-derivative-local} gives
\begin{equation}\label{eq:good-perron-height}
\frac{\zeta_F'}{\zeta_F}(\sigma+i(\tau\pm U))
\ll_F\log^2T.
\end{equation}
By \eqref{eq:good-perron-height}, each horizontal integral is
\begin{equation*}
\ll_F\frac{\log^2T}{U}\int_{-1/2}^{c}z^\sigma\,d\sigma
\ll_F\frac zT\log^2T.
\end{equation*} 
The functional equation and
Stirling's formula give
\begin{equation*}
\frac{\zeta_F'}{\zeta_F}\left(-\frac12+iv\right)
\ll_F\log(|v|+3).
\end{equation*}
Hence the integral over the new vertical segment satisfies
\begin{equation*}
\ll_Fz^{-1/2}\int_{-U}^{U}
\frac{\log(|\tau+v|+3)}{1+|v|}\,dv
\ll_Fz^{-1/2}\log^2T.
\end{equation*}
Together with the Perron error, these estimates establish the error term in \eqref{eq:twisted-explicit-formula}. It remains to bound the explicit contributions from the logarithmic derivative and the nontrivial zeros.

Applying  \eqref{eq:log-derivative-local} with $\Re s=0$, we obtain
\begin{equation*}
\frac{\zeta_F'}{\zeta_F}(i\tau)\ll_F\log^2T.
\end{equation*}
Since $U\asymp T$, the pole at $s=1$ contributes $O(z/T)$, while the possible zero at $s=0$ contributes $O_F(1/T)$.  Under GRH and applying
\eqref{eq:unit-zero-count}, we obtain
\begin{equation*}
\sum_{\substack{\rho_F=1/2+i\gamma_F\\
|\gamma_F-\tau|<U}} \frac{z^{\rho_F-i\tau}}{\rho_F-i\tau}\ll 
\sum_{\substack{\rho_F=1/2+i\gamma_F\\
|\gamma_F-\tau|<U}}\frac{z^{1/2}}{1+|\gamma_F-\tau|}
\ll_Fz^{1/2}\log T\sum_{0\le j\le U}\frac1{1+j}
\ll_Fz^{1/2}\log^2T.
\end{equation*}
Substitution in \eqref{eq:twisted-explicit-formula} proves
\eqref{eq:twisted-chebyshev-bound}.
\end{proof}

\begin{lem}\label{dedekind-square-tail}
Let $K$ be a fixed number field, let $L/\mathbb Q$ be a fixed finite Galois extension containing $K$, and assume GRH for $\zeta_L(s)$.  There is $T_0=T_0(K,L)\ge3$ such that, for every $T\ge T_0$, uniformly for $T/2\le t\le5T/2$ and $e^2\le x\le T^2$,
\begin{equation*}
\Re\sum_{\log T<p\le\sqrt x}
\frac{\Lambda_K(p^2)}{2p^{1+2it}\log p}
\frac{\log(x/p^2)}{\log x}\,p^{-2/\log x}
=O_{K,L}(1).
\end{equation*}
\end{lem}
\begin{proof}
Let $H=\Gal(L/K)$, so that $K=L^H$, and set
\begin{equation*}
        Y=(\log T)^{10},
        \qquad
        X=\sqrt x,
        \qquad
        \tau=2t,
\end{equation*}
and
\begin{equation*}
        \omega_x(u)
        =u^{-2/\log x}\frac{\log(x/u^2)}{\log x}.
\end{equation*}
For $1\le u\le X$ one has $0\le\omega_x(u)\le1$.  Hence, using
$|\Lambda_K(p^2)|\ll_K\log p$ and Mertens' estimate,
\begin{equation*}
        \sum_{\log T<p\le\min\{Y,X\}}
        \frac{|\Lambda_K(p^2)|}{p\log p}\omega_x(p)
        \ll_K
        \sum_{\log T<p\le Y}\frac1p
        \ll_K1.
\end{equation*}
If $X\le4Y$, the remaining primes lie in $Y<p\le4Y$ and contribute
$O_K(1)$.  We may therefore suppose that $X>4Y$.

After enlarging $T_0$, every rational prime ramified in $L$ is smaller than
$\log T$.  Thus \eqref{eq:artin-square-prime} applies to every prime in the
remaining sum.  Moreover, for each cyclic subgroup $C_\nu$ occurring in
Lemma~\ref{artin-square-coefficients}, the extension $L/L^{C_\nu}$ is
Galois.  By the Aramata--Brauer theorem \cite{Aramata,Brauer}, GRH for
$\zeta_L$ implies GRH for $\zeta_{L^{C_\nu}}$.  It therefore suffices to
prove that, for each fixed field $F=L^{C_\nu}$,
\begin{equation}\label{eq:square-tail-auxiliary-field}
        \sum_{Y<p\le X}
        \frac{\Lambda_F(p)}{2p^{1+i\tau}\log p}\omega_x(p)
        \ll_F1.
\end{equation}

Lemma~\ref{lem:twisted-chebyshev} gives
\eqref{eq:twisted-chebyshev-bound} uniformly for $2\le z\le T$.
Let
\begin{equation*}
        W_x(u)=\frac{\omega_x(u)}{2u\log u}.
\end{equation*}
Since $X^2=x$, it follows that $W_x(X)=0$.  Also, for $Y\le u\le X$,
\begin{equation*}
        \omega_x'(u)
        =
        -\frac4{u\log x}
        u^{-2/\log x}
        \left(1-\frac{\log u}{\log x}\right),
\end{equation*}
and hence
\begin{equation}\label{eq:square-tail-weight-bounds}
        W_x(u)\ll\frac1{u\log u},
        \qquad
        W_x'(u)\ll\frac1{u^2\log u}.
\end{equation}
Partial summation, followed by
\eqref{eq:twisted-chebyshev-bound} and
\eqref{eq:square-tail-weight-bounds}, gives
\begin{align*}
&\left|
\sum_{Y<n\le X}\Lambda_F(n)n^{-i\tau}W_x(n)
\right| \\
&\qquad\le
|\Psi_F(Y,\tau)W_x(Y)|
+\int_Y^X|\Psi_F(u,\tau)W_x'(u)|\,du \\
&\qquad\ll_F
\frac{\log^2T}{\sqrt Y\log Y}
+\frac1{T\log Y}
+\log^2T\int_Y^X\frac{\,du}{u^{3/2}\log u}
+\frac1T\int_Y^X\frac{\,du}{u\log u} \\
&\qquad\ll_F
\frac{\log^2T}{\sqrt Y\log Y}
+\frac{\log\log T}{T}
\ll_F1,
\end{align*}
because $Y=(\log T)^{10}$ and $X\le T$.

It remains to remove the contribution of the terms for which $n$ is a proper prime power.  Since
$|\Lambda_F(p^a)|\ll_F\log p$ and $|\omega_x(u)|\le1$ for $u\le X$,
\begin{align*}
\sum_{\substack{Y<p^a\le X\\a\ge2}}
\left|
\Lambda_F(p^a)(p^a)^{-i\tau}W_x(p^a)
\right|
&\ll_F
\sum_{\substack{p^a>Y\\a\ge2}}
\frac1{a p^a}\ll_FY^{-1/2}.
\end{align*}
Thus
\begin{equation*}
        \sum_{Y<p\le X}
        \frac{\Lambda_F(p)}{2p^{1+i\tau}\log p}\omega_x(p)
        \ll_F1,
\end{equation*}
which proves \eqref{eq:square-tail-auxiliary-field}.

Finally, applying \eqref{eq:artin-square-prime}, summing over the finitely many
auxiliary fields $L^{C_\nu}$, and combining this with the contribution from the already treated
range $\log T<p\le Y$, gives
\begin{equation*}
        \sum_{\log T<p\le X}
        \frac{\Lambda_K(p^2)}{2p^{1+i\tau}\log p}\omega_x(p)
        \ll_{K,L}1.
\end{equation*}
Since $X=\sqrt x$ and $\tau=2t$, taking real parts proves the lemma.
\end{proof}

	\begin{lem}\label{conditional-upper-bound}
			Let $L/\mathbb Q$ be finite Galois and let $K$ be an intermediate field. Assume GRH for $\zeta_L(s)$. There is $T_0=T_0(K,L)\ge3$ such that, for $T\ge T_0$, uniformly for $T/2\le t\le5T/2$ and $e^2\le x\le T^2$, we have
		\begin{align*}
			&\log \left| \zeta_K\left(\tfrac{1}{2}+it\right)\right|
				\le  \Re\left( \sum_{p\le  x} \frac{\Lambda_K(p)}{p^{\frac{1}{2}+\frac{1}{\log x}+ it}\log p} \frac{\log\left(x/p\right)}{\log x}+ \sum_{p\le  \min\{ \sqrt{x}, \ \log T\}} \frac{\Lambda_K(p^2)}{2p^{1+2it}\log p}\right)\\
					& \hspace{3.5cm}+ \frac{[K:\mathbb{Q}]\log T}{\log x} +O_{K,L}(1),
		\end{align*}
		where $p$ denotes the rational prime.
	\end{lem}
	\begin{proof}
The Aramata--Brauer theorem shows that $\zeta_L/\zeta_K$ is entire, so GRH for $\zeta_L$ implies GRH for $\zeta_K$.  Choosing $\lambda=1$ in Lemma~\ref{conditional-bound-chandee}, and using $\log\mathcal C_K(t)=[K:\mathbb Q]\log T+O_K(1)$ in the stated range, we see that
		\begin{equation*}
			\log \left| \zeta_K\left(\tfrac{1}{2}+it\right)\right|\le  \Re\sum_{n\le  x}\frac{\Lambda_K(n)}{n^{\frac{1}{2}+\frac{1}{\log x}+it}\log n}\frac{\log\left(x/n\right)}{\log x} + \frac{[K:\mathbb{Q}]\log T}{\log x} +O_K\!\left(\tfrac{1}{\log x}\right).
		\end{equation*}
		Since $|\Lambda_K(n)|\le  [K:\mathbb{Q}]\Lambda(n)$ and $\log(x/n)/(n^{1/\log x}\log x)\le  1$ for $n\le  x$, we have
		\begin{equation*}
			\sum_{\substack{{p^j\le  x},\\ 
					j\geqslant 3}} \frac{\Lambda_K(p^j)}{(p^j)^{\frac{1}{2}+\frac{1}{\log x}+it}\log p^j}\frac{\log\left(x/p^j\right)}{\log x} \ll [K:\mathbb{Q}] 	\sum_{\substack{{p^j\le  x}, \\ 
					j\geqslant 3}} \frac{1}{jp^{j/2}} \ll 1.
		\end{equation*}
It remains to justify that the contribution of the prime-square sum with $\log T<p\le \sqrt x$ is $O_{K,L}(1)$.  By Lemma~\ref{dedekind-square-tail},
\begin{equation*}
        \Re\sum_{\log T<p\le  \sqrt x}
        \frac{\Lambda_K(p^2)}{2p^{1+2it}\log p}
        \frac{\log(x/p^2)}{\log x}\,p^{-2/\log x}
        \ll_{K,L}1.
\end{equation*}
For the remaining prime-square terms, the weight in Chandee's formula may be replaced by $1$, at a cost of $O_K(1)$.  Indeed, for $p\le\sqrt x$ one has
\begin{equation*}
        \left|1-\frac{\log(x/p^2)}{\log x}p^{-2/\log x}\right|
        \ll \frac{\log p}{\log x},
\end{equation*}
and therefore
\begin{equation*}
        \sum_{p\le\min\{\sqrt x,\log T\}}
        \frac{|\Lambda_K(p^2)|}{p\log p}
        \left|1-\frac{\log(x/p^2)}{\log x}p^{-2/\log x}\right|
        \ll_K
        \frac1{\log x}\sum_{p\le\sqrt x}\frac{\log p}{p}
        \ll_K1.
\end{equation*}
Retaining the prime-square terms with $p\le\min\{\sqrt x,\log T\}$ and absorbing the rest into the $O_{K,L}(1)$ term completes the lemma.
	\end{proof}
    The following proposition combines the Chandee-type majorants for the fields under consideration.
    
	\begin{prop}\label{main-prop} 
Let $K_1,\ldots, K_r$ be fixed number fields, let $L$ be the compositum of their Galois closures, and assume GRH for $\zeta_L(s)$. Also, let $a_1,a_2,\ldots,a_r$  be fixed positive real numbers, and let $a_K$ be a fixed real number satisfying $a_K>\sum_{i=1}^r a_i[K_i:\mathbb{Q}]$,
and let $b_1,b_2,\ldots,b_r$ be real numbers satisfying
$|b_i|\le  T/2$ for $1\le  i\le  r$. For any positive integer $n$, define
\begin{equation}\label{eq:def-h}
\mathfrak{h}(n):=\frac{a_1\Lambda_{K_1}(n)n^{-ib_1}+\cdots+a_r\Lambda_{K_r}(n)n^{-ib_r}}{2}.
\end{equation}
The elementary bound $|\Lambda_{K_i}(n)|\le[K_i:\mathbb Q]\Lambda(n)$ gives, for every rational prime $p$,
\begin{equation}\label{eq:h-prime-uniform-bounds}
        |\mathfrak h(p)|\le\frac12\sum_{i=1}^ra_i[K_i:\mathbb Q]\log p,
        \qquad
        |\mathfrak h(p^2)|\le\frac12\sum_{i=1}^ra_i[K_i:\mathbb Q]\log p.
\end{equation}
These bounds are uniform in the shifts.
There is $T_0=T_0(K_1,\ldots,K_r,L,\mathbf a,a_K)\ge3$ such that, for $T\ge T_0$, every $t\in[T,2T]$, and every $e^2\le x\le T^2$, \begin{align*}
			&a_1\log\left| \zeta_{K_1}\left(1/2+i(t+b_1)\right)\right| +\cdots + 	a_r\log\left| \zeta_{K_r}\left(1/2+i(t+b_r)\right)\right| \\
			& \le  2\Re\left( \sum_{p\le  x} \frac{\mathfrak{h}(p)}{p^{\frac{1}{2}+\frac{1}{\log x}+it}\log p} \frac{\log\left(x/p\right)}{\log x}+ \sum_{p\le  \min\{ \sqrt{x},\log T\}} \frac{\mathfrak{h}(p^2)}{2p^{1+2it}\log p}\right) + a_K \frac{\log T}{\log x} +O_{K_1,\ldots,K_r,L,\mathbf a}(1).
		\end{align*}
	\end{prop}
	\begin{proof}
Apply Lemma~\ref{conditional-upper-bound} to each field $K_i$, using the common Galois extension $L$, at the shifted point $t+b_i$, and then multiply the resulting inequality by $a_i$ and sum over $i$.  Since
\begin{equation*}
        \sum_{i=1}^ra_i\Lambda_{K_i}(n)n^{-ib_i}=2\mathfrak h(n),
\end{equation*}
the prime and prime-square terms combine exactly into the two sums appearing in the statement.  The conductor term is bounded by $a_K\log T/\log x$, because $a_K$ was chosen larger than $\sum_i a_i[K_i:\mathbb Q]$. Therefore,
		\begin{align*}
		&	\sum_{i=1}^r a_i \log|\zeta_{K_i}\left(\tfrac{1}{2}+i(t+b_i)\right)|\le \Re\sum_{p\le  x}	\sum_{i=1}^r  \frac{a_ip^{-ib_i}\Lambda_{K_i}(p)}{p^{\frac{1}{2}+\frac{1}{\log x}+ it}\log p} \frac{\log\left(x/p\right)}{\log x}\\
			&+ \Re\sum_{p \le  \min\left\{\sqrt{x}, \log T\right\}}	\sum_{i=1}^r  \frac{a_i p^{-2ib_i}\Lambda_{K_i}(p^2)}{2p^{1+2it}\log p}
		+\frac{\log T}{\log x}\sum_{i=1}^ra_i[K_i:\mathbb{Q}] +O(1).
		\end{align*}
	Hence the proposition follows.
	\end{proof}
    
	\subsection{A crude upper bound for shifted moments}
In the proof of Theorem \ref{main-theorem}, assuming GRH, we shall need a crude upper bound for shifted moments of $L$-functions on the critical line.
\begin{lem}\label{shifted-moment-trivial-bound}
Let $K_1,\ldots,K_m$ be fixed number fields, let $L/\mathbb Q$ be a finite Galois extension containing them, and let $A_1,\ldots,A_m>0$ be fixed.  Assume GRH for $\zeta_L(s)$.  Uniformly for $|b_i|\le T/2$, there is a constant $C$, depending only on $L$, the fields, and the exponents, such that
		\begin{align*}
			\int_{T}^{2T} \prod_{j=1}^{m}\left| \zeta_{K_j}\left(\frac{1}{2}+i(t+b_j)\right)\right|^{A_j}\,dt\ll T(\log T)^C.
		\end{align*}
	\end{lem}
	\begin{proof}
It suffices to prove, for each fixed intermediate field $K$,  $A>0$ and $U\ge 3$ that
\begin{equation}\label{eq:one-field-crude-moment}
\int_U^{2U}\left|\zeta_K\left(\frac12+iu\right)\right|^A\,du
        \ll_{L,K,A} U(\log U)^{C(L,K,A)}.
\end{equation}
Indeed, GRH for $\zeta_L$ implies GRH for $\zeta_K$ by the Aramata--Brauer theorem.  Let $d=[K:\mathbb Q]$ and, for $V\ge3$, define
\begin{equation*}
        \mathcal E(V)=
        \left\{u\in[U,2U]:
        \log\left|\zeta_K\left(\frac12+iu\right)\right|>V\right\}.
\end{equation*}
We first establish the two estimates needed below.  There exist positive constants $C_1,C_2,C_3,D>0$, depending only on $K$ and $A$, such that, with
\begin{equation*}
        V_0=D\frac{\log U}{\log\log U},
\end{equation*}
one has
\begin{align}
\label{eq:dedekind-large-values-small}
        \operatorname{meas}(\mathcal E(V))
        &\ll
        U\exp\left(-\frac{V^2}{C_1\log\log U}\right)
         &\text{for }3\le  V\le  C_2\log\log U,\\
\label{eq:dedekind-large-values-large}
        \operatorname{meas}(\mathcal E(V))
        &\ll
        U\exp\left(-C_3V\log\frac{V}{\log\log U}\right)
        &\text{for }C_2\log\log U\le V\le V_0,
\end{align}
while $\mathcal E(V)=\emptyset$ for $V>V_0$.  To prove these assertions, let $x=e^Y$, where $2\le Y\le2\log U$.  By Lemma~\ref{conditional-upper-bound} and the bound $|\Lambda_K(n)|\le d\Lambda(n)$,
\begin{equation*}
        \sum_{p\le  \log U}\frac{|\Lambda_K(p^2)|}{p\log p}
        \ll_K \sum_{p\le\log U}\frac1p
        \ll\log\log\log U
\end{equation*}
gives, uniformly for $u\in[U,2U]$,
\begin{equation}\label{eq:sound-prime-majorant}
        \log\left|\zeta_K\left(\frac12+iu\right)\right|
        \le
        \Re P_x(u)+\frac{d\log U}{Y}+O_K(\log\log\log U),
\end{equation}
where
\begin{equation*}
        P_x(u)=
        \sum_{p\le  x}
        \frac{\Lambda_K(p)}{p^{1/2+1/Y+iu}\log p}
        \frac{\log(x/p)}{Y}.
\end{equation*}
If $r\ge1$ and $x^r\le U^{1/10}$, Lemma~\ref{complex-dirichlet-mean} gives
\begin{equation}\label{eq:Px-moment}
        \int_U^{2U}|P_x(u)|^{2r}\,du
        \ll
        U r!\left(C_K\log\log x\right)^r
        +U^{1/10}(C_Kx^{1/2})^{2r}
        \ll
        U r!\left(C_K\log\log x\right)^r+U^{1/5}C_K^{2r}.
\end{equation}
After division by $(V/2)^{2r}$, the off-diagonal term is at most
\begin{equation*}
        U^{1/5}\left(\frac{2C_K}{V}\right)^{2r}.
\end{equation*}
We now choose $Y=8d\log U/V$.  When $C_K\log\log\log U\le V\le C_2\log\log U$, membership in $\mathcal E(V)$ implies $\Re P_x(u)>V/2$.  With
\begin{equation*}
        r=\left\lfloor\frac{V^2}{C_4\log\log U}\right\rfloor+1,
\end{equation*}
and $C_4$ sufficiently large, we have $x^r\le U^{1/10}$, and Chebyshev's inequality gives
\begin{equation*}
        \operatorname{meas}(\mathcal E(V))
        \ll
        U\left(\frac{C_Kr\log\log U}{V^2}\right)^r
        \ll
        U\exp\left(-\frac{V^2}{C_1\log\log U}\right).
\end{equation*}
The range $3\le V<C_K\log\log\log U$ follows from $\operatorname{meas}(\mathcal E(V))\le U$ after enlarging $C_1$.  This proves \eqref{eq:dedekind-large-values-small}.  Lemma~\ref{conditional-upper-bound} with $x=(\log U)^2$ also gives
\begin{align*}
\log\left|\zeta_K\left(\frac12+iu\right)\right|
&\le d\sum_{p\le(\log U)^2}\frac1{\sqrt p}
+\frac{d\log U}{2\log\log U}+O_K(\log\log\log U)\\
&\le D\frac{\log U}{\log\log U}
\end{align*}
for a fixed $D=D(K)$; hence $\mathcal E(V)=\emptyset$ for $V>V_0$.  If $C_2\log\log U<V\le V_0$, keep the previous choice of $Y$ and take
\begin{equation*}
        r=\left\lfloor\frac{V}{C_5d}\right\rfloor,
\end{equation*}
where $C_5\ge80$ is fixed.  Then $Y\ge(8d/D)\log\log U$ and
\begin{equation*}
        rY\le\frac{8}{C_5}\log U,
\end{equation*}
so $x^r\le U^{1/10}$.  The diagonal term in \eqref{eq:Px-moment} gives
\begin{equation*}
        \operatorname{meas}(\mathcal E(V))
        \ll
        U\left(\frac{C_Kr\log\log U}{V^2}\right)^r
        \ll
        U\exp\left(-C_3V\log\frac{V}{\log\log U}\right),
\end{equation*}
after $C_2$ is enlarged.  In both ranges the off-diagonal contribution is $O(U^{1/5})$.  Taking $C_3D\le1/2$ makes the right side of \eqref{eq:dedekind-large-values-large} at least $U^{1/2}$, so this error is absorbed and \eqref{eq:dedekind-large-values-large} follows.

We now integrate the distribution estimates.  With $W(u)=\log|\zeta_K(1/2+iu)|$,
\begin{equation*}
        \int_U^{2U}e^{AW(u)}\,du
        \le
        Ue^{AC_2\log\log U}
        +A\int_{C_2\log\log U}^{\infty}
        e^{AV}\operatorname{meas}(\mathcal E(V))\,dV.
\end{equation*}
Choose $C_2$ so that $C_3\log C_2\ge A+2$.  The integral is then $O(U)$, proving \eqref{eq:one-field-crude-moment}.  Finally, let $A=A_1+\cdots+A_m$.  Hölder's inequality gives
\begin{equation*}
\int_T^{2T}\prod_{j=1}^m
\left|\zeta_{K_j}\left(\frac12+i(t+b_j)\right)\right|^{A_j}\,dt
\le
\prod_{j=1}^m
\left(
\int_{T+b_j}^{2T+b_j}
\left|\zeta_{K_j}\left(\frac12+iu\right)\right|^{A}\,du
\right)^{A_j/A}.
\end{equation*}
Each shifted interval $[T+b_j,2T+b_j]$ can be covered by $O(1)$ dyadic intervals whose endpoints are comparable to $T$.  Applying \eqref{eq:one-field-crude-moment} on each such interval completes the proof.
\end{proof}

\subsection{Frobenius coefficients and Euler products}
We now establish Chebotarev-type identities that record the cancellations arising from the Galois-theoretic structure of the fields. These identities are fundamental in handling the non-Galois case and allow us to relate prime sums to logarithms of suitable Euler products.

\begin{lem}\label{euler-product-comparison}
Let $K_1,\ldots,K_r$, $L$, and the characters $\chi_i$ be those defined before Theorem~\ref{main-theorem}, let $S$ be the finite set of rational primes ramified in $L$, and let $1\le i,j\le r$.  Uniformly for $x\geqslant3$ and real $u$,
\begin{equation*}
        \sum_{\substack{p\le x\\p\notin S}}
        \frac{\chi_i(\Frob_p)\chi_j(\Frob_p)\cos(u\log p)}p
        =
        \log\left|\calZ_{ij}\left(1+\frac1{\log x}+iu\right)\right|
        +O_{K_1,\ldots,K_r,L}(1).
\end{equation*}
Here the logarithm of each Dedekind zeta factor in $\calZ_{ij}$ is the branch obtained from the absolutely convergent Euler product in $\Re s>1$.  Ramified primes in the common Galois closure, and all prime-power terms $p^r$ with $r\ge2$, contribute $O_{K_1,\ldots,K_r,L}(1)$, uniformly in $x$ and $u$.
\end{lem}
\begin{proof}
For $\Re s>1$, we use the logarithm determined by the absolutely convergent Euler product, so that $\log\zeta_F(s)$ is real for real $s>1$.  If $p\notin S$, then the primes of $K_i$ above $p$ are the orbits of $\langle\Frob_p\rangle$ on $G/H_i$, and an orbit has residue degree equal to its length.  Thus the coefficient of $p^{-s}$ in $\log\zeta_{K_i}(s)$ is the number of one-element orbits, namely $\chi_i(\Frob_p)$.

The product action of $G$ on $G/H_i\times G/H_j$ has character $\chi_i\chi_j$.  Its orbits are indexed by the double cosets $H_i\backslash G/H_j$, and the stabiliser of $(H_i,gH_j)$ is $H_i\cap gH_jg^{-1}$.  Hence the coefficient of $p^{-s}$ in the logarithmic Euler product for
\begin{equation*}
        \calZ_{ij}(s)=
        \prod_{H_i gH_j\in H_i\backslash G/H_j}
        \zeta_{L^{H_i\cap gH_jg^{-1}}}(s)
\end{equation*}
is the number of fixed points of $\Frob_p$ on $G/H_i\times G/H_j$, which is $\chi_i(\Frob_p)\chi_j(\Frob_p)$. Consequently, for $\Re s>1$,
\begin{equation*}
        \log\calZ_{ij}(s)
        =
        \sum_{p\notin S}\frac{\chi_i(\Frob_p)\chi_j(\Frob_p)}{p^s}
        +O_{K_1,\ldots,K_r,L}(1).
\end{equation*}
Here the $O(1)$ term is uniform for $\Re s\ge1$.  Indeed, the contribution from the ramified primes is a finite product depending only on $L$, and the terms $p^{-ms}$, $m\ge2$, are bounded in absolute value by
\begin{equation*}
        C\sum_p\sum_{m\ge2}p^{-m\Re s}\ll 1,
\end{equation*}
with $C$ depending only on the fixed fields.

To compare the Euler product with the truncated prime sum, set $s=1+\delta+iu$, where $\delta=1/\log x$.  The coefficients $\chi_i(\Frob_p)\chi_j(\Frob_p)$ are bounded.  Mertens' estimate and partial summation give
\begin{equation*}
        \sum_{p\le  x}\frac{1-p^{-\delta}}p+
        \sum_{p>x}\frac1{p^{1+\delta}}
        \ll 1,
\end{equation*}
from which it follows that
\begin{equation*}
        \sum_{p\notin S}\frac{\chi_i(\Frob_p)\chi_j(\Frob_p)}{p^{1+\delta+iu}}
        =
        \sum_{\substack{p\le x\\p\notin S}}
        \frac{\chi_i(\Frob_p)\chi_j(\Frob_p)}{p^{1+iu}}
        +O_{K_1,\ldots,K_r,L}(1),
\end{equation*}
uniformly in real $u$.  Taking real parts gives the lemma.
\end{proof}

For the proof of the simultaneous theorem, we use the following Frobenius-orbit description, which amounts to a local calculation at each unramified prime.
\begin{prop}\label{frobenius-orbit-product}
Uniformly for $x\geq3$ and real shifts $b_1,\ldots,b_r$,
\begin{equation*}
        \sum_{p\le  x}\frac{|\mathfrak h(p)|^2}{p\log^2p}
        =
        \log B(x,\mathbf b,\mathbf a)+O(1).
\end{equation*}
In particular, if $3\le  x\le  T$, then
\begin{equation*}
        \exp\left(\sum_{p\le  x}\frac{|\mathfrak h(p)|^2}{p\log^2p}\right)
        \ll_{\mathbf a,\mathbf K} B(T,\mathbf b,\mathbf a).
\end{equation*}
\end{prop}
\begin{proof}
Let $S$ be the finite set of rational primes ramified in $L$.  Since $|\mathfrak h(p)|\ll_{\mathbf a,\mathbf K}\log p$, the primes in $S$ contribute $O(1)$ to the sum.  For $p\notin S$, the degree-one coefficient identity gives
\begin{equation*}
        \frac{\Lambda_{K_i}(p)}{\log p}=\chi_i(\Frob_p).
\end{equation*}
Thus
\begin{equation*}
        \frac{|\mathfrak h(p)|^2}{\log^2p}
        =
        \frac14\sum_{i,j=1}^r
        a_i a_j\chi_i(\Frob_p)\chi_j(\Frob_p)p^{-i(b_i-b_j)}.
\end{equation*}
The ordered sum is real, because the terms with indices $(i,j)$ and $(j,i)$ are complex conjugates and the diagonal terms are real.  Applying Lemma~\ref{euler-product-comparison} with $u=b_i-b_j$, multiplying by $a_i a_j/4$, and summing over $i,j$, gives
\begin{equation*}
        \sum_{p\le  x}\frac{|\mathfrak h(p)|^2}{p\log^2p}
        =
        \sum_{i,j=1}^r\frac{a_i a_j}{4}
        \log\left|\calZ_{ij}\left(1+\frac1{\log x}+i(b_i-b_j)\right)\right|+O(1).
\end{equation*}
This is exactly $\log B(x,\mathbf b,\mathbf a)+O(1)$.  The final assertion follows from positivity of the left hand side: if $3\le  x\le  T$, then
\begin{equation*}
        \sum_{p\le  x}\frac{|\mathfrak h(p)|^2}{p\log^2p}
        \le
        \sum_{p\le  T}\frac{|\mathfrak h(p)|^2}{p\log^2p}
        =
        \log B(T,\mathbf b,\mathbf a)+O(1).
\end{equation*}
Exponentiating completes the proof.
\end{proof}

\begin{lem}\label{B-poly-bounds}
There is a constant $C_B\ge1$, depending only on the fields and on $\mathbf a$, such that
\begin{equation*}
        C_B^{-1}(\log T)^{-C_B}
        \le B(T,\mathbf b,\mathbf a)
        \le C_B(\log T)^{C_B}
\end{equation*}
uniformly for $|b_i|\le  T/2$.
\end{lem}
\begin{proof}
Each factor in $B$ is a Dedekind zeta function of a fixed intermediate field, evaluated at $s=1+1/\log T+i u$.  Since $\zeta_F(s)$ has an Euler product with positive local degrees,
\begin{equation*}
        |\zeta_F(s)|\le  \zeta_F(1+1/\log T)\ll_F \log T
\end{equation*}
and, from the reciprocal Euler product,
\begin{equation*}
        |\zeta_F(s)|^{-1}\le  \zeta_F(1+1/\log T)\ll_F\log T.
\end{equation*}
Thus every zeta factor occurring in $B$ lies between two fixed powers of $\log T$.  There are only finitely many such factors and the exponents $a_i a_j/4$ are fixed, so the claimed bounds follow.
\end{proof}

		\section{Setup and proof of Theorem~\ref{main-theorem}}\label{sec:parameter-accounting}
We now turn to the proof of Theorem \ref{main-theorem}. The proof is based on Harper's method \cite{Harper}, together with some modifications needed for our setting.
	We define an increasing geometric sequence $\left(\alpha_i\right)_{i\geq0}$ of real numbers by $\alpha_0:= 0$ and for all $i\geqslant 1$, \begin{equation*}
		\alpha_i:= \frac{20^{i-1}}{(\log \log T)^2}.
	\end{equation*}
For the compact set $\mathcal A$ in Theorem~\ref{main-theorem}, set
    \begin{equation*}
        a_K=50+\max_{\mathbf a\in\mathcal A}
        \sum_{q=1}^ra_q[K_q:\mathbb Q].
    \end{equation*}
Then \eqref{eq:h-prime-uniform-bounds}, the truncations, the exceptional-set
estimates, Lemma~\ref{shifted-moment-trivial-bound}, and the Euler-product
comparisons are uniform on $\mathcal A$.
We enlarge $C_B$ in Lemma~\ref{B-poly-bounds}, if necessary, so that its
bounds also hold uniformly for $\mathbf a\in\mathcal A$.
Let
    \begin{equation*}
        R_0=\frac{1}{1-20^{-1/4}},
        \qquad
        C_0=400a_K^2,
        \qquad
        C=4C_0=1600a_K^2.
    \end{equation*}
Fix a sufficiently large constant $a'\ge\max\{a_K^2,1\}$ satisfying
    \begin{align}
200R_0\sqrt{a'}\left(20e^{-1000a'}\right)^{1/4}
        \le\frac1{100}, &\qquad   20e^{-1000a'}\le\frac{1}{200a_K^2},\nonumber\\
\sqrt{a'}\left(20e^{-1000a'}\right)^{-3/4}\ge2, \qquad  & \frac{1000a'-\log20}{20C}\ge a_K+1.
  \label{eq:fixed-parameter-conditions}
    \end{align}
    Let $\theta_i:=\sqrt{a'}\alpha_i^{-3/4}$.  After increasing the lower
    bound on $T$, assume $\alpha_1\le e^{-1000a'}$ and define
    \begin{equation*}
    \mathcal J:=1+\max\{i\ge1:\alpha_i\le e^{-1000a'}\}.
    \end{equation*}
    Then
    \begin{equation}\label{eq:terminal-alpha-block}
    \alpha_{\mathcal J-1}\le e^{-1000a'}
    <\alpha_{\mathcal J}=20\alpha_{\mathcal J-1}
    \le20e^{-1000a'}.
    \end{equation}
    In particular,
    \begin{equation*}
\theta_i\ge\theta_{\mathcal J}\ge
            \sqrt{a'}\left(20e^{-1000a'}\right)^{-3/4}>1.
    \end{equation*}
    Consequently, for $1\le i\le\mathcal J$,
    \begin{equation}\label{eq:all-truncation-lengths}
            \lfloor100\theta_i\rfloor\ge10\theta_i,
            \qquad
            \lfloor200\theta_i\rfloor\ge20\theta_i.
    \end{equation}
    Moreover $\theta_i=20^{3/4}\theta_{i+1}$.  It follows that
    \begin{equation}\label{eq:truncation-errors-summable}
    \sum_{i=1}^{\mathcal J}e^{-9\theta_i}
    \le
    \sum_{k\ge0}\exp\left(-9\sqrt{a'}
    \left(20e^{-1000a'}\right)^{-3/4}20^{3k/4}\right)
    \ll_{a'}1.
    \end{equation}
    Since $C$ dominates the absolute constant in
    Lemma~\ref{truncate-dirichlet-poly}, multiplication over any first
    $j\le\mathcal J$ blocks contributes at most
    \begin{equation}\label{eq:truncation-error-product}
    \prod_{i=1}^{j}\left(1+Ce^{-9\theta_i}\right)
    \le\exp\left(C\sum_{i=1}^{j}e^{-9\theta_i}\right)
    \ll_{a'}1.
    \end{equation}
    For doubled block polynomials the analogous factor is
    $\exp(O(\sum_i e^{-18\theta_i}))\ll_{a'}1$.  We also use
    \begin{equation}\label{eq:geometric-support-sum}
        \sum_{i=1}^j\alpha_i\theta_i
        =\sqrt{a'}\sum_{i=1}^j\alpha_i^{1/4}
        \le R_0\sqrt{a'}\alpha_j^{1/4}
        \le R_0\sqrt{a'}\alpha_{\mathcal J}^{1/4}.
    \end{equation}
By enlarging the lower bound on $T$ if necessary, we may assume that, uniformly for
    $j\le\mathcal J$,
    \begin{align}
     &\sum_{i=1}^j\alpha_i\lfloor100\theta_i\rfloor
            \le\frac1{200},\qquad
        \sum_{i=1}^j\alpha_i\lfloor200\theta_i\rfloor
            \le\frac1{100},\nonumber\\
       & \frac{(2m+2)\log2}{\log T}\lfloor2^{3m/4}\rfloor
            \le\frac1{100}
            \qquad \text{for }0\le m\le\frac{\log\log T}{\log2},\nonumber\\
        &\left(\prod_{i=1}^j
        \sum_{0\le n\le\lfloor200\theta_i\rfloor}
        \frac1{n!}\left(2a_K\sum_{p\le T^{\alpha_i}}p^{-1/2}\right)^n
        \right)^2\le T^{1/50}.
        \label{eq:numerical-support-bounds}
    \end{align}
    The first two estimates follow from \eqref{eq:geometric-support-sum}.
    The third follows on taking logarithms.  For the last estimate,
    $\sum_{p\le y}p^{-1/2}\ll y^{1/2}/\log y$ gives, after increasing the
    lower bound on $T$,
    \begin{align*}
    &\left(\prod_{i=1}^j
    \sum_{0\le n\le\lfloor200\theta_i\rfloor}
    \frac1{n!}\left(2a_K\sum_{p\le T^{\alpha_i}}p^{-1/2}\right)^n
    \right)^2\\
    &\qquad\le
    \left(\prod_{i=1}^j(\lfloor200\theta_i\rfloor+1)\right)^2
    T^{\sum_{i=1}^j\alpha_i\lfloor200\theta_i\rfloor}
    \le T^{1/50},
    \end{align*}
    because the logarithm of the squared product of the truncation lengths
    is $O_{a'}((\log\log\log T)^2)\le(\log T)/100$.  We shall also assume
    \begin{equation}\label{eq:square-root-range}
            \sqrt{T^{\alpha_j}}\ge\sqrt{T^{\alpha_1}}\ge\log T.
    \end{equation}
    Finally, let $N=\lfloor2^{3m/4}\rfloor$. We require that
    \begin{align}
        a_K^{2N}2^{mN/5}\exp\left(10a_K2^{m/2}\right)
        &\le T^{1/100}
        \qquad \text{for }0\le m\le\frac{2\log\log\log T}{\log2},\nonumber\\
        a_K^{2N}2^{mN/5}
        &\le T^{1/10} \qquad \text{for }
        0\le m\le\frac{\log\log T}{\log2}.
        \label{eq:numerical-square-growth}
    \end{align}
   The first inequality follows from
   \begin{equation*}
   O_{a_K}((\log\log T)^{3/2}\log\log\log T)
   \le\frac{\log T}{100},
   \end{equation*}
   and the second follows from
   \begin{equation*}
   O_{a_K}((\log T)^{3/4}\log\log T)
   \le\frac{\log T}{10}.
   \end{equation*}
   Finally, increase the lower bound on $T$ so that
   \begin{equation}\label{eq:explicit-error-absorption}
   \frac4C(\log\log T)^2\log\log\log T
   +C_B\log\log T+\log(C_B\mathcal J)
   \le\frac14\log T.
   \end{equation}
    For $2\le n\le x$, we set \begin{equation*}
		\Theta_x(n):= \frac{1}{n^{1/\log x}\log n}\frac{\log(x/n)}{\log x}.
	\end{equation*}
	 Recall from Proposition~\ref{main-prop} and \eqref{eq:def-h} the definition of $\mathfrak h$. For $1\le  i\le  j\le  \mathcal{J}$, let
	\begin{equation}\label{def-pij}
		\mathcal{P}_{i,j}(t):= \sum_{T^{\alpha_{i-1}}< p \le  T^{\alpha_i}} \frac{\Theta_{T^{\alpha_j}}(p)}{p^{1/2+it}}\mathfrak{h}(p).
	\end{equation}
	We cover $[T,2T]$ by the sets $\mathcal{T},\mathcal{S}(0),\ldots,
    \mathcal{S}(\mathcal{J}-1)$, where
	\begin{equation*}
		\mathcal{T}= \mathcal{T}_{\mathbf{a},\mathbf K, T}:= \left\{ t\in [T,2T]: \left| \mathcal{P}_{i,\mathcal{J}}(t)\right|\le  \theta_i, \ \forall\, 1\le  i \le  \mathcal{J}\right\},
	\end{equation*}
	\begin{equation*}
		\mathcal{S}(0)= 	\mathcal{S}_{\mathbf{a},\mathbf K,T}(0):=\left\{ t\in [T,2T]: \left| \mathcal{P}_{1, \ell}(t)\right|> \theta_1 \ \text{for some}\ 1\le  \ell \le  \mathcal{J}\right\},
	\end{equation*}
	and
	for $1\le  j \le  \mathcal{J}-1$,
	\begin{align*}
		\mathcal{S}(j)=\mathcal{S}_{\mathbf{a},\mathbf K,T}(j):= \left\{ t\in[T,2T]: \begin{array}{l}
				\left| \mathcal{P}_{i,\ell}(t)\right|\le  \theta_i \ \forall\,1\le  i\le  j, \forall \, i\le \ell\le  \mathcal{J}\\
				\text{but}\ 	\left| \mathcal{P}_{j+1,\ell}(t)\right|> \theta_{j+1} \ \text{for some} \ j+1\le  \ell\le  \mathcal{J}
		\end{array}\right\}.
	\end{align*}
	Observe that 
	\begin{equation}\label{integral-splits}
		\left[T, 2T\right] =\bigcup_{j=0}^{\mathcal{J}-1}\mathcal{S}(j) \cup \mathcal{T}.
	\end{equation}
	Indeed, if an admissible truncation fails, choose the least block index at
	which this occurs; otherwise the point lies in $\mathcal T$.
	The cover need not be disjoint: $\mathcal T$ controls only
    $\mathcal P_{i,\mathcal J}$, whereas membership in $\mathcal S(j)$ may be
    caused by $\mathcal P_{j+1,\ell}$ with $\ell<\mathcal J$.  This causes no
    difficulty because the integrand is non-negative.  For every
    $t\in\mathcal S(j)$ and $j\ge1$ we have, exactly,
    \begin{equation*}
    |\mathcal P_{i,\ell}(t)|\le\theta_i
    \quad(1\le i\le j,\ i\le\ell\le\mathcal J),
    \qquad
    |\mathcal P_{j+1,\ell_0}(t)|>\theta_{j+1}
    \end{equation*}
    for at least one $\ell_0\in[j+1,\mathcal J]$; for
    $t\in\mathcal S(0)$ the latter assertion holds with $j=0$.
	Therefore, it suffices to prove that 
	\begin{align*}
	\nonumber	&\sum_{j=0}^{\mathcal{J}-1}\int_{t\in \mathcal{S}(j)} \prod_{i=1}^r\left| \zeta_{K_i}\left(1/2+i(t+b_i)\right)\right|^{a_i}\,dt +\int_{t\in\mathcal{T}} \prod_{i=1}^r\left| \zeta_{K_i}\left(1/2+i(t+b_i)\right)\right|^{a_i}\,dt\\
		& \ \ \ \ \ll T B(T,\mathbf b,\mathbf a).
	\end{align*}
To prove the above bound, we require the following three lemmas, which will be proved in the next section.
	\begin{lem}\label{first-main-lemma}
	For the fixed fields $K_1,\ldots,K_r$, and for sufficiently large $T$,
		\begin{equation*}
			\int_{t\in\mathcal{T}} \exp\left(2 \Re \sum_{p\le  T^{\alpha_\mathcal{J}}} \frac{\Theta_{T^{\alpha_{\mathcal{J}}}}(p)}{p^{1/2+it}}\mathfrak{h}(p)\right)\,dt \ll T B(T,\mathbf b,\mathbf a),
		\end{equation*}
		where the implicit constant depends only on the fields and on $\mathcal A$.
	\end{lem}
	\begin{lem}\label{second-main-lemma}
	For the fixed fields $K_1,\ldots,K_r$, one has $\text{meas}(\mathcal{S}(0))\ll T\exp(-c(\log\log T)^2)$, and for any $1\le  j\le  \mathcal{J}-1$,
		\begin{align*}
			\int_{t\in\mathcal{S}(j)} &\exp\left(2 \Re \sum_{p\le  T^{\alpha_j}} \frac{\Theta_{T^{\alpha_{j}}}(p)}{p^{1/2 +it}}\mathfrak{h}(p)\right)\,dt \\
            &\ll T B(T,\mathbf b,\mathbf a)
            \exp\left( -\frac{\log(1/\alpha_{j+1})}{C\alpha_{j+1}}\right),
		\end{align*}
			where $C=1600a_K^2$ is fixed above.
			Here $c>0$ and all implied constants depend only on the fields
			and on $\mathcal A$.
	\end{lem}
	
	\begin{lem}\label{third-main-lemma}
		With the notation and assumptions above,
		\begin{align}
		\int_{t\in\mathcal{T}} &\exp\left(2 \Re \left(\sum_{p\le  T^{\alpha_\mathcal{J}}} \frac{\Theta_{T^{\alpha_{\mathcal{J}}}}(p)}{p^{1/2 +it}}\mathfrak{h}(p)+\frac{1}{2}\sum_{p\le  \log T} \frac{\mathfrak{h}(p^2)}{p^{1+2it}\log p}\right)\right)\,dt\nonumber\\
			&\ll_{\mathcal A,\mathbf K} T B(T,\mathbf b,\mathbf a).
		\label{first-integral-third-lemma}
		\end{align}
		and for $1\le  j\le  \mathcal{J}-1$, we have 
		\begin{align}\label{second-integral-third-lemma}
			\int_{t\in\mathcal{S}(j)}&\exp\left(2 \Re \left(\sum_{p\le  T^{\alpha_j}} \frac{\Theta_{T^{\alpha_j}}(p)}{p^{1/2 +it}}\mathfrak{h}(p)+\frac{1}{2}\sum_{p\le  \log T} \frac{\mathfrak{h}(p^2)}{p^{1+2it}\log p}\right)\right)\,dt\\
			& \nonumber\ll_{\mathcal A,\mathbf K}T B(T,\mathbf b,\mathbf a)\exp\left( -\frac{\log(1/\alpha_{j+1})}{C\alpha_{j+1}}\right).
		\end{align}
		\end{lem}
		\subsection{Proof of Theorem~\ref{main-theorem}}
\begin{proof}
Since $|b_i|\le T/2$, Proposition~\ref{main-prop} applies uniformly to each shifted zeta function. For $t\in \mathcal{T}$ we apply Proposition~\ref{main-prop} with $x=T^{\alpha_{\mathcal J}}$.  By \eqref{eq:square-root-range}, we have $\sqrt{x}\ge \log T$ so the prime-square contribution is truncated at $\log T$. By \eqref{eq:terminal-alpha-block},
the conductor contribution satisfies $a_K/\alpha_{\mathcal J}\le a_Ke^{1000a'}$.  The first assertion of
Lemma~\ref{third-main-lemma} therefore gives the required bound on
$\mathcal T$.

For $1\le j \le \mathcal{J}-1$ and $t\in \mathcal S(j)$, take $x=T^{\alpha_j}$.  The second
assertion of Lemma~\ref{third-main-lemma} gives
\begin{equation*}
 \int_{t\in \mathcal S(j)}\prod_{i=1}^r
 \left|\zeta_{K_i}\left(\frac12+i(t+b_i)\right)\right|^{a_i}\,dt
 \ll TB(T,\mathbf b,\mathbf a)e^{a_K/\alpha_j}
 \exp\left(-\frac{\log(1/\alpha_{j+1})}{C\alpha_{j+1}}\right).
\end{equation*}
The last condition in \eqref{eq:fixed-parameter-conditions}, together with
$\alpha_{j+1}=20\alpha_j$, bounds the product of two exponential factors by
$e^{-1/\alpha_j}$.  Since $\alpha_{\mathcal J-1-k}
=\alpha_{\mathcal J-1}/20^k$ and $\alpha_{\mathcal J-1}\le e^{-1000a'}$,
\begin{equation*}
\sum_{j=1}^{\mathcal J-1}e^{-1/\alpha_j}
\le\sum_{k\ge0}\exp(-20^ke^{1000a'})\ll_{a'}1.
\end{equation*}

Finally, applying the Cauchy--Schwarz inequality together with Lemmas~\ref{second-main-lemma} and \ref{shifted-moment-trivial-bound} (the latter with exponents $2a_i$), we obtain
\begin{equation*}
 \int_{t\in \mathcal S(0)}\prod_{i=1}^r
 \left|\zeta_{K_i}\left(\frac12+i(t+b_i)\right)\right|^{a_i}\,dt
 \ll T(\log T)^D e^{-c(\log\log T)^2}.
\end{equation*}
By Lemma~\ref{B-poly-bounds}, the right side is
$O(TB(T,\mathbf b,\mathbf a))$ after increasing the lower bound on $T$.
The decomposition \eqref{integral-splits} completes the proof.
\end{proof}
		\section{Proofs of the main lemmas}
    The proof is the Harper good-set argument, with the prime diagonal evaluated by Proposition~\ref{frobenius-orbit-product}.
	
\subsection{Proof of Lemma~\ref{first-main-lemma}}

\begin{proof}
For $1\le i\le\mathcal J$, write
\begin{equation*}
 \mathcal P_i(t)=\mathcal P_{i,\mathcal J}(t)
 =\sum_{T^{\alpha_{i-1}}<p\le T^{\alpha_i}}c_i(p)p^{-it},
 \qquad c_i(p)=\frac{\Theta_{T^{\alpha_{\mathcal J}}}(p)\mathfrak h(p)}{\sqrt p},
\end{equation*}
and let $N_i=\lfloor100\theta_i\rfloor$. For $t\in \mathcal T$, we have
$|\mathcal P_i(t)|\le\theta_i$ and $N_i\ge10\theta_i$.  Hence
Lemma~\ref{truncate-dirichlet-poly} and \eqref{eq:truncation-errors-summable}
give
\begin{align}\label{eq:first-lemma-truncation}
 \int_{t\in \mathcal T}\prod_{i\le\mathcal J}e^{2\Re\mathcal P_i(t)}\,dt
 \ll\int_T^{2T}\left|\prod_{i\le\mathcal J}
 \sum_{n\le N_i}\frac{\mathcal P_i(t)^n}{n!}\right|^2\,dt.
\end{align}
The support is at most $T^{\sum_i\alpha_iN_i}\le T^{1/200}$.
Lemma~\ref{complex-dirichlet-mean}, together with the coefficient-$L^1$ bound in \eqref{eq:numerical-support-bounds}, gives
\begin{align}\label{eq:first-lemma-mean}
 \int_{t\in \mathcal T}\prod_{i\le\mathcal J}e^{2\Re\mathcal P_i(t)}\,dt
 \ll T\exp\left(\sum_{i\le\mathcal J}\sum_{T^{\alpha_{i-1}}<p\le T^{\alpha_i}}
 |c_i(p)|^2\right)+T^{1/5}.
\end{align}
Indeed, the off-diagonal term is at most $T^{1/10}T^{1/50}$, while
$|c_i(p)|^2\le|\mathfrak h(p)|^2/(p\log^2p)$.  Applying Proposition~\ref{frobenius-orbit-product} and Lemma~\ref{B-poly-bounds} to \eqref{eq:first-lemma-mean} completes the proof.
\end{proof}

\subsection{Proof of Lemma~\ref{second-main-lemma}}

\begin{proof}
We first consider $1\le  j\le\mathcal J-1$.  For $1\le  i\le  j$, let
\begin{equation*}
        c_{i,j}(p)=\frac{\Theta_{T^{\alpha_j}}(p)\mathfrak h(p)}{\sqrt p},
        \qquad
        \mathcal P_{i,j}(t)=
        \sum_{T^{\alpha_{i-1}}<p\le  T^{\alpha_i}}c_{i,j}(p)p^{-it},
\end{equation*}
and let $N_i=\lfloor100\theta_i\rfloor$.  Since every $t\in\mathcal S(j)$ satisfies $|\mathcal P_{i,j}(t)|\le\theta_i$ for $1\le i\le j$, \eqref{eq:all-truncation-lengths} gives $N_i\ge10\theta_i$.  Lemma~\ref{truncate-dirichlet-poly}, with \eqref{eq:truncation-errors-summable}, gives
\begin{align}\label{eq:second-lemma-after-truncation}
        &\int_{t\in \mathcal S(j)}
        \exp\left(2\Re\sum_{p\le  T^{\alpha_j}}
        \frac{\Theta_{T^{\alpha_j}}(p)\mathfrak h(p)}{p^{1/2+it}}\right)\,dt        \\
        &\qquad\ll
        \sum_{\ell=j+1}^{\mathcal J}
        \int_T^{2T}
        \left|
        \prod_{i=1}^{j}\sum_{0\le  n\le  N_i}\frac{\mathcal P_{i,j}(t)^n}{n!}
        \right|^2
        \mathbf{1}_{\{|\mathcal P_{j+1,\ell}(t)|>\theta_{j+1}\}}\,dt.
\end{align}
Let
\begin{equation*}
        M=\left\lfloor\frac{1}{100a_K^2\alpha_{j+1}}\right\rfloor.
\end{equation*}
For the indicator in \eqref{eq:second-lemma-after-truncation} we use
\begin{equation*}
        \mathbf{1}_{\{|\mathcal P_{j+1,\ell}(t)|>\theta_{j+1}\}}
        \le
        \theta_{j+1}^{-2M}|\mathcal P_{j+1,\ell}(t)|^{2M}.
\end{equation*}
We now apply Lemma~\ref{complex-dirichlet-mean} with prime blocks $T^{\alpha_{i-1}}<p\le  T^{\alpha_i}$, $1\le  i\le  j$, and with
\begin{equation*}
        E_1(t)=\mathcal P_{j+1,\ell}(t),
        \qquad
        \mathcal Q_1=\{p:T^{\alpha_j}<p\le  T^{\alpha_{j+1}}\}.
\end{equation*}
The support condition in that lemma holds, since every integer which occurs is at most
\begin{equation*}
        \left(\prod_{i=1}^{j}T^{\alpha_iN_i}\right)T^{\alpha_{j+1}M}
        \le
        T^{1/200+1/(100a_K^2)}
        \le  T^{1/10}
\end{equation*}
by \eqref{eq:numerical-support-bounds}.  Let
\begin{equation*}
        V_{j,\ell}:=
        \sum_{T^{\alpha_j}<p\le  T^{\alpha_{j+1}}}
        \frac{\Theta_{T^{\alpha_\ell}}(p)^2|\mathfrak h(p)|^2}{p}.
\end{equation*}
Applying Lemma~\ref{complex-dirichlet-mean}, we obtain
\begin{align}\label{eq:second-lemma-mean-bound}
        &\int_T^{2T}
        \left|
        \prod_{i=1}^{j}\sum_{0\le  n\le  N_i}\frac{\mathcal P_{i,j}(t)^n}{n!}
        \right|^2
        |\mathcal P_{j+1,\ell}(t)|^{2M}\,dt                         \\
        &\qquad\ll
        T\exp\left(
        \sum_{p\le  T^{\alpha_j}}\frac{|\mathfrak h(p)|^2}{p\log^2p}
        \right)M!V_{j,\ell}^M
        +O(T^{1/2}).
\end{align}
For the off-diagonal term, \eqref{eq:numerical-support-bounds} bounds the
first squared coefficient-$L^1$ product by $T^{1/50}$.  Moreover
\begin{equation*}
        \sum_{T^{\alpha_j}<p\le  T^{\alpha_{j+1}}}
        \frac{\Theta_{T^{\alpha_\ell}}(p)|\mathfrak h(p)|}{\sqrt p}
        \le
        \frac{a_K}{2}\sum_{p\le  T^{\alpha_{j+1}}}p^{-1/2}
        \ll
        \frac{a_KT^{\alpha_{j+1}/2}}
        {\alpha_{j+1}\log T}
        \le T^{\alpha_{j+1}/2}.
\end{equation*}
The last inequality is uniform because
$\alpha_{j+1}\ge(\log\log T)^{-2}$.  Raising this estimate to the power
$2M$ gives
\begin{equation}\label{eq:exceptional-l1-bound}
        \left(
        \sum_{T^{\alpha_j}<p\le T^{\alpha_{j+1}}}
        \frac{\Theta_{T^{\alpha_\ell}}(p)|\mathfrak h(p)|}{\sqrt p}
        \right)^{2M}
        \le T^{\alpha_{j+1}M}
        \le T^{1/(100a_K^2)}
        \le T^{1/100}.
\end{equation}
Hence the whole off-diagonal contribution is at most
$    T^{1/10}T^{1/50}T^{1/100}=T^{13/100}\le T^{1/4}$.
Multiplication by the factor $\theta_{j+1}^{-2M}$ can only decrease this contribution.

For the diagonal, \eqref{eq:h-prime-uniform-bounds} gives
$|\mathfrak h(p)|<a_K\log p/2$.  Mertens' estimate gives
\begin{equation*}
        \sum_{T^{\alpha_j}<p\le T^{\alpha_{j+1}}}\frac1p
        \le\log20+1<4
\end{equation*}
uniformly for $j<\mathcal J$, after increasing the lower bound on $T$.
Since $\Theta_x(p)\le1/\log p$, it follows that
\begin{equation*}
        V_{j,\ell}
        \le
        \frac{a_K^2}{4}\sum_{T^{\alpha_j}<p\le  T^{\alpha_{j+1}}}\frac1p
        \le a_K^2.
\end{equation*}
By the second condition in \eqref{eq:fixed-parameter-conditions},
\begin{equation*}
        \frac{1}{100a_K^2\alpha_{j+1}}\ge2,
        \qquad
        M\ge\frac{1}{200a_K^2\alpha_{j+1}}.
\end{equation*}
Using $M!\le M^M$ and $\theta_{j+1}^2=a'\alpha_{j+1}^{-3/2}$, we obtain
\begin{align}\label{eq:second-lemma-saving}
        \theta_{j+1}^{-2M}M!V_{j,\ell}^M
        &\le
        \left(\frac{Ma_K^2}{\theta_{j+1}^2}\right)^M
        \le
        \left(\frac{\alpha_{j+1}^{1/2}}{100a'}\right)^M                                      \\
        &\le
        \exp\left(-\frac{\log(1/\alpha_{j+1})}{C_0\alpha_{j+1}}\right).
\end{align}
Proposition~\ref{frobenius-orbit-product} gives
\begin{equation*}
        \exp\left(
        \sum_{p\le  T^{\alpha_j}}\frac{|\mathfrak h(p)|^2}{p\log^2p}
        \right)
        \ll B(T,\mathbf b,\mathbf a).
\end{equation*}
We now sum uniformly over $\ell$. By the geometric definition of the blocks,
\begin{equation*}
        \alpha_{\mathcal J}
        =20^{\mathcal J-j-1}\alpha_{j+1},
        \qquad
        \mathcal J-j
        =1+\frac{\log(\alpha_{\mathcal J}/\alpha_{j+1})}{\log20}
        \le
        2+\frac{\log(1/\alpha_{j+1})}{\log20}.
\end{equation*}
Here $\alpha_{\mathcal J}\le20e^{-1000a'}<1$.  Also
$C_0\alpha_{j+1}\le C_0\alpha_{\mathcal J}\le2$ by the second
condition in \eqref{eq:fixed-parameter-conditions}.  Since
$\log(1/\alpha_{j+1})\ge1000a'-\log20$, our fixed choice of $a'$ gives
\begin{equation*}
        \log(\mathcal J-j)
        \le\frac14\log(1/\alpha_{j+1})
        \le
        \frac{3\log(1/\alpha_{j+1})}{4C_0\alpha_{j+1}}.
\end{equation*}
Consequently, since $C=4C_0$,
\begin{equation}\label{eq:uniform-ell-sum}
        (\mathcal J-j)
        \exp\left(-\frac{\log(1/\alpha_{j+1})}{C_0\alpha_{j+1}}\right)
        \le
        \exp\left(-\frac{\log(1/\alpha_{j+1})}{C\alpha_{j+1}}\right).
\end{equation}
Combining \eqref{eq:second-lemma-after-truncation}, \eqref{eq:second-lemma-mean-bound}, \eqref{eq:second-lemma-saving}, and \eqref{eq:uniform-ell-sum} yields
\begin{equation*}
        \int_{t\in\mathcal S(j)}
        \exp\left(2\Re\sum_{p\le  T^{\alpha_j}}
        \frac{\Theta_{T^{\alpha_j}}(p)\mathfrak h(p)}{p^{1/2+it}}\right)\,dt
        \ll
        T B(T,\mathbf b,\mathbf a)
        \exp\left(-\frac{\log(1/\alpha_{j+1})}{C\alpha_{j+1}}\right).
\end{equation*}
It remains to absorb the $O(T^{1/2})$ terms.  Since
$\alpha_1\le\alpha_{j+1}<1$, we have
\begin{equation*}
        \frac{\log(1/\alpha_{j+1})}{C\alpha_{j+1}}
        \le
        \frac2C(\log\log T)^2\log\log\log T.
\end{equation*}
Lemma~\ref{B-poly-bounds} and
\eqref{eq:explicit-error-absorption} therefore give
\begin{equation*}
        T B(T,\mathbf b,\mathbf a)
        \exp\left(-\frac{\log(1/\alpha_{j+1})}{C\alpha_{j+1}}\right)
        \ge \mathcal J T^{3/4},
\end{equation*}
which absorbs the sum of all these errors.

It remains to bound $\operatorname{meas}(\mathcal S(0))$.  If $t\in\mathcal S(0)$, then $|\mathcal P_{1,\ell}(t)|>\theta_1$ for some $1\le\ell\le\mathcal J$.  With
\begin{equation*}
        M_0=\left\lfloor\frac{1}{100a_K^2\alpha_1}\right\rfloor
\end{equation*}
we apply Lemma~\ref{complex-dirichlet-mean} with $R=0$, $E_1(t)=\mathcal P_{1,\ell}(t)$, and $M_1=M_0$.  Its support is at most
\begin{equation*}
        T^{\alpha_1M_0}\le T^{1/(100a_K^2)}\le T^{1/10}.
\end{equation*} The diagonal estimate together with the coefficient $L^1$-bound for the off-diagonal contribution gives
\begin{align*}
        \operatorname{meas}(\mathcal S(0))
        &\le
        \sum_{\ell=1}^{\mathcal J}\theta_1^{-2M_0}
        \int_T^{2T}|\mathcal P_{1,\ell}(t)|^{2M_0}\,dt                         \\
        &\ll
        \sum_{\ell=1}^{\mathcal J}
        \theta_1^{-2M_0}
        \left(T M_0!V_{0,\ell}^{M_0}+T^{1/2}\right),
\end{align*}
where
\begin{equation*}
        V_{0,\ell}=\sum_{p\le  T^{\alpha_1}}
        \frac{\Theta_{T^{\alpha_\ell}}(p)^2|\mathfrak h(p)|^2}{p}
        \le a_K^2\log\log T
\end{equation*}
after increasing the lower bound on $T$.  Moreover,
\begin{equation*}
\frac{(\log\log T)^2}{200a_K^2}
\le M_0\le\frac{(\log\log T)^2}{100a_K^2}.
\end{equation*}
Since $\theta_1^2=a'(\log\log T)^3$ and $M_0!\le M_0^{M_0}$,
\begin{equation*}
        \theta_1^{-2M_0}M_0!V_{0,\ell}^{M_0}
        \le
        \left(\frac1{100a'}\right)^{M_0}
        \le\exp(-c(\log\log T)^2)
\end{equation*}
for a fixed $c=c(a_K,a')>0$.  Moreover, the geometric definition gives
\begin{equation*}
        \mathcal J
        =1+\frac{\log(\alpha_{\mathcal J}/\alpha_1)}{\log20}
        \ll_{a'}\log\log\log T,
\end{equation*}
and hence, after decreasing $c$ once,
\begin{equation*}
        \mathcal J\exp(-c(\log\log T)^2)
        \ll \exp(-c(\log\log T)^2).
\end{equation*}
It follows that
\begin{equation*}
        \operatorname{meas}(\mathcal S(0))
        \ll
        T\exp(-c(\log\log T)^2).
\end{equation*}
Here the sum of the $T^{1/2}$ terms is smaller than the above bound once
$\log\mathcal J+c(\log\log T)^2\le\frac12\log T$.
The proof of Lemma~\ref{second-main-lemma} is complete.
\end{proof}

\subsection{Proof of Lemma~\ref{third-main-lemma}}

\begin{proof}
Let
\begin{equation*}
 M_T=\lceil\log_2\log T\rceil-1,
 \qquad I_m=\{p:2^m<p\le\min\{2^{m+1},\log T\}\}
 \qquad(0\le m\le M_T),
\end{equation*}
and define
\begin{equation*}
 R(t)=\sum_{p\le\log T}\frac{\mathfrak h(p^2)}{2p^{1+2it}\log p},
 \qquad R_m(t)=\sum_{p\in I_m}\frac{\mathfrak h(p^2)}{2p^{1+2it}\log p}.
\end{equation*}
Thus $R=\sum_{m=0}^{M_T}R_m$, where the last interval is truncated at $\log T$.
The coefficient bound \eqref{eq:h-prime-uniform-bounds} yields the estimates, including the last interval,
\begin{equation}\label{eq:square-block-norms}
 \sum_{p\in I_m}\frac{|\mathfrak h(p^2)|^2}{4p^2\log^2p}
 \le C\,2^{-m},
 \qquad
 \sum_{p\in I_m}\frac{|\mathfrak h(p^2)|}{2p\log p}\le a_K.
\end{equation}

Let $\mathcal R(m)$ be the set of $t\in[T,2T]$ such that
\begin{equation*}
 |R_m(t)|>2^{-m/10},\qquad |R_n(t)|\le2^{-n/10}\quad(m<n\le M_T),
\end{equation*}
and let
\begin{equation*}
 \mathcal R_*=\{t\in[T,2T]:|R_m(t)|\le2^{-m/10}
 \text{ for every }0\le m\le M_T\}.
\end{equation*}
The sets $\mathcal R_*,\mathcal R(0),\ldots,\mathcal R(M_T)$ are disjoint
and cover $[T,2T]$: in $\mathcal R(m)$, $m$ is the largest index for which the inequality fails.

Fix $0\le m\le M_T$ and set $N=\lfloor2^{3m/4}\rfloor$.  The indicator inequality
\begin{equation}\label{eq:square-block-indicator}
 \mathbf1_{\mathcal R(m)}(t)
 \le2^{mN/5}|R_m(t)|^{2N}
\end{equation}
and Lemma~\ref{complex-dirichlet-mean}, applied with the pairwise coprime
auxiliary integers $q=p^2$ and coefficients
$\mathfrak h(p^2)/(2p\log p)$, give, for some fixed $c>0$,
\begin{equation}\label{eq:square-block-measure}
 \operatorname{meas}(\mathcal R(m))
 \ll T2^{mN/5}N!(C\,2^{-m})^N+T^{1/2}
 \ll T e^{-c\,2^{3m/4}}+T^{1/2}.
\end{equation}
Here $(2^{2m+2})^N\le T^{1/100}$, and the off-diagonal term is
$T^{1/10}\exp(O(m2^{3m/4}))=O(T^{1/2})$.  Moreover, for $t\in\mathcal R(m)$,
\begin{equation}\label{eq:square-block-pointwise}
 |\Re R(t)|\ll\sum_{p\le2^{m+1}}p^{-1}
 +\sum_{n>m}2^{-n/10}\ll\log(m+2)+1.
\end{equation}

We now estimate the contribution jointly with the main Dirichlet polynomial.
Fix a constant $D>0$, depending only on the fields and $\mathcal A$, large
enough to dominate the constants below.  Write
\begin{equation*}
 Q_k(t)=\sum_{p\le T^{\alpha_k}}
 \frac{\Theta_{T^{\alpha_k}}(p)\mathfrak h(p)}{p^{1/2+it}},
 \qquad E_j=\exp\left(-\frac{\log(1/\alpha_{j+1})}{C\alpha_{j+1}}\right).
\end{equation*}
Suppose $0\le m\le2\log\log\log T/\log2$.  Remove from $Q_k$ the terms corresponding to primes $p\le2^{m+1}$ and denote the resulting polynomial by $Q_k^{(m)}$.
In this range,
\begin{equation*}
2^{m+1}\le2(\log\log T)^2<T^{\alpha_1},
\end{equation*}
so the deletion affects only the first prime block.  Moreover,
\begin{equation*}
\sum_{p\le2^{m+1}}
\frac{\Theta_{T^{\alpha_k}}(p)|\mathfrak h(p)|}{\sqrt p}
\le2a_K2^{m/2}\le2a_K\log\log T\le\theta_1
\end{equation*}
for all sufficiently large $T$.  Thus, on $\mathcal T$ with
$k=\mathcal J$ and on $\mathcal S(j)$ with $k=j$, the altered first
block is at most $2\theta_1$ and every other block is at most $\theta_i$.
Combining
\eqref{eq:square-block-indicator} with \eqref{eq:square-block-pointwise}, we
obtain pointwise on either set $X$,
\begin{equation}\label{eq:square-joint-pointwise}
 \mathbf1_{\mathcal R(m)}e^{2\Re(Q_k+R)}
 \ll e^{D2^{m/2}}2^{mN/5}|R_m|^{2N}e^{2\Re Q_k^{(m)}}.
\end{equation}
Truncate the exponential of every block in $Q_k^{(m)}$ at
$\lfloor200\theta_i\rfloor$.  If the product of these truncated polynomials
is $A_{k,m}(t)$, Lemma~\ref{truncate-dirichlet-poly} changes the last factor
in \eqref{eq:square-joint-pointwise} to $|A_{k,m}(t)|^2$ up to a bounded
factor.  Its support together with $R_m^N$ is at most
\begin{equation*}
 T^{\sum_i\alpha_i\lfloor200\theta_i\rfloor}(2^{2m+2})^N\le T^{1/50}.
\end{equation*}
Lemma~\ref{complex-dirichlet-mean}, again using the auxiliary integers
$q=p^2$, with $E_1=R_m$ and $M_1=N$, therefore gives
\begin{align}\label{eq:square-joint-diagonal}
 &2^{mN/5}\int_T^{2T}|R_m(t)^NA_{k,m}(t)|^2\,dt\nonumber\\
 &\quad\ll T\exp\left(\sum_{2^{m+1}<p\le T^{\alpha_k}}
 \frac{|\mathfrak h(p)|^2}{p\log^2p}\right)
 2^{mN/5}N!(C\,2^{-m})^N+O(T^{1/4}).
\end{align}
The error follows from \eqref{eq:numerical-support-bounds},
\eqref{eq:numerical-square-growth}, and \eqref{eq:square-block-norms}.  Also
\begin{equation}\label{eq:square-block-saving}
 e^{D2^{m/2}}2^{mN/5}N!(C\,2^{-m})^N
 \ll e^{D2^{m/2}-c\,2^{3m/4}}.
\end{equation}

We record the additional calculation for $X=\mathcal S(j)$, where
$1\le j\le\mathcal J-1$.  For each
$t\in\mathcal S(j)$ choose the least $\ell\in[j+1,\mathcal J]$ for which
$|\mathcal P_{j+1,\ell}(t)|>\theta_{j+1}$; summing over all $\ell$ is an
upper bound because the integrand is non-negative.  Let
\begin{equation*}
 M=\left\lfloor\frac1{100a_K^2\alpha_{j+1}}\right\rfloor,
 \qquad
 V_{j,\ell}=\sum_{T^{\alpha_j}<p\le T^{\alpha_{j+1}}}
 \frac{\Theta_{T^{\alpha_\ell}}(p)^2|\mathfrak h(p)|^2}{p}.
\end{equation*}
We insert
$\theta_{j+1}^{-2M}|\mathcal P_{j+1,\ell}(t)|^{2M}$ in
\eqref{eq:square-joint-pointwise}.  The three sets of variables are pairwise
coprime: $R_m^N$ is supported on the integers $p^2$ with $p\le2^{m+1}$,
$A_{j,m}$ uses primes $2^{m+1}<p\le T^{\alpha_j}$, and
$\mathcal P_{j+1,\ell}^M$ uses primes
$T^{\alpha_j}<p\le T^{\alpha_{j+1}}$.  Their combined support is at most
\begin{equation*}
 T^{\sum_{i\le j}\alpha_i\lfloor200\theta_i\rfloor}
 (2^{2m+2})^N T^{\alpha_{j+1}M}
 \le T^{1/100+1/100+1/(100a_K^2)}\le T^{1/10}.
\end{equation*}
Lemma~\ref{complex-dirichlet-mean} therefore gives the complete diagonal
contribution
\begin{align*}
 T&\exp\left(\sum_{2^{m+1}<p\le T^{\alpha_j}}
 \frac{|\mathfrak h(p)|^2}{p\log^2p}\right)
 2^{mN/5}N!\left(C\,2^{-m}\right)^N
 \theta_{j+1}^{-2M}M!V_{j,\ell}^M.
\end{align*}
The last factor is bounded by \eqref{eq:second-lemma-saving}.  For the
off-diagonal term, the squared coefficient sums of $A_{j,m}$, $R_m^N$, and
$\mathcal P_{j+1,\ell}^M$, including both indicator factors, are bounded by
$T^{1/50}$, $T^{1/100}$, and $T^{1/100}$, respectively, by
\eqref{eq:numerical-support-bounds}, \eqref{eq:numerical-square-growth}, and
\eqref{eq:exceptional-l1-bound}.  Thus the off-diagonal is
$O(T^{1/10+1/50+1/100+1/100})=O(T^{7/50})$, uniformly in $m,j,\ell$.
Summing the diagonal estimates over $\ell$ and using
\eqref{eq:uniform-ell-sum} gives the factor $E_j$.  The off-diagonal
estimates do not give this factor.  They are absorbed separately: in the
present range, $2^{3m/4}\le(\log\log T)^{3/2}$, and we may increase the
lower bound on $T$ so that
\begin{equation*}
c(\log\log T)^{3/2}\le\frac12\log T.
\end{equation*}
The lower bound following \eqref{eq:explicit-error-absorption} then gives
\begin{equation*}
\mathcal J T^{7/50}e^{D2^{m/2}}
\le
TB(T,\mathbf b,\mathbf a)E_j
e^{D2^{m/2}-c\,2^{3m/4}}.
\end{equation*}
The same calculation absorbs the $O(T^{1/4})$ error in the good-set
estimate.
Consequently Proposition~\ref{frobenius-orbit-product},
Lemma~\ref{B-poly-bounds}, and \eqref{eq:square-block-saving} yield
\begin{align}
 \int_{t\in\mathcal T\cap\mathcal R(m)}e^{2\Re(Q_{\mathcal J}+R)}\,dt
 &\ll TB(T,\mathbf b,\mathbf a)
 e^{D2^{m/2}-c\,2^{3m/4}},\label{eq:square-small-good}\\
 \int_{t\in \mathcal S(j)\cap\mathcal R(m)}e^{2\Re(Q_j+R)}\,dt
 &\ll TB(T,\mathbf b,\mathbf a)E_j
 e^{D2^{m/2}-c\,2^{3m/4}}.
 \label{eq:square-small-bad}
\end{align}

For the remaining $m$, truncate $2Q_{\mathcal J}$ or $2Q_j$ at
$\lfloor200\theta_i\rfloor$.  The calculation without $R_m$ gives
\begin{align}
 \int_{t\in \mathcal T}e^{4\Re Q_{\mathcal J}}\,dt&\ll T(\log T)^D,
 \label{eq:doubled-good}\\
 \int_{t\in \mathcal S(j)}e^{4\Re Q_j}\,dt&\ll T(\log T)^D E_j^2.
 \label{eq:doubled-bad}
\end{align}
For \eqref{eq:doubled-bad} use, for each $\ell$, the indicator power
\begin{equation*}
 M_2=\left\lfloor\frac1{50a_K^2\alpha_{j+1}}\right\rfloor.
\end{equation*}
The combined support is at most
$T^{1/100+1/(50a_K^2)}\le T^{1/10}$, and the prime blocks are disjoint.
Since $V_{j,\ell}\le a_K^2$, the exceptional diagonal factor satisfies
\begin{align*}
 \theta_{j+1}^{-2M_2}M_2!V_{j,\ell}^{M_2}
 &\le\left(\frac{M_2a_K^2}{\theta_{j+1}^2}\right)^{M_2}
 \le\left(\frac{\alpha_{j+1}^{1/2}}{50a'}\right)^{M_2} \\
 &\le \exp\left(-\frac{2\log(1/\alpha_{j+1})}
 {C_0\alpha_{j+1}}\right).
\end{align*}
Here the fixed parameter conditions ensure
$M_2\ge(100a_K^2\alpha_{j+1})^{-1}$.  The coefficient-$L^1$ bounds used
above give an off-diagonal $O(T^{1/4})$.  Finally
\eqref{eq:uniform-ell-sum}, with the stronger exponent in the last equation,
shows that summing over $\ell$ costs at most
\begin{equation*}
 \exp\left(-\frac{2\log(1/\alpha_{j+1})}
 {C\alpha_{j+1}}\right)=E_j^2.
\end{equation*}
Since $\alpha_{j+1}\ge(\log\log T)^{-2}$,
\begin{equation*}
E_j^2\ge
\exp\left(-\frac4C(\log\log T)^2\log\log\log T\right).
\end{equation*}
Thus \eqref{eq:explicit-error-absorption} implies
$TE_j^2\ge\mathcal J T^{3/4}$, which absorbs the off-diagonal sum.
This proves \eqref{eq:doubled-bad}; the same mean-value calculation without
an exceptional factor proves \eqref{eq:doubled-good}.

For $m>2\log\log\log T/\log2$, the Cauchy--Schwarz inequality gives
\begin{align*}
&\int_{t\in \mathcal T\cap\mathcal R(m)}
e^{2\Re(Q_{t\in \mathcal J}+R)}\,dt\le
\left(\int_{t\in \mathcal T}e^{4\Re Q_{\mathcal J}}\,dt\right)^{1/2}
\left(\int_{t\in \mathcal R(m)}e^{4\Re R}\,dt\right)^{1/2},
\\
&\int_{t\in \mathcal S(j)\cap\mathcal R(m)}
e^{2\Re(Q_j+R)}\,dt \le
\left(\int_{t\in \mathcal S(j)}e^{4\Re Q_j}\,dt\right)^{1/2}
\left(\int_{t\in \mathcal R(m)}e^{4\Re R}\,dt\right)^{1/2}.
\end{align*}
By \eqref{eq:square-block-pointwise},
$e^{4\Re R(t)}\ll(m+2)^D$ on $\mathcal R(m)$.  Hence
\eqref{eq:square-block-measure} and
\eqref{eq:doubled-good}--\eqref{eq:doubled-bad} give
\begin{align*}
&\int_{t\in \mathcal T\cap\mathcal R(m)}
e^{2\Re(Q_{\mathcal J}+R)}\,dt\ll(m+2)^D(\log T)^D
\left(Te^{-(c/2)2^{3m/4}}+T^{3/4}\right),
\\
&\int_{t\in \mathcal S(j)\cap\mathcal R(m)}
e^{2\Re(Q_j+R)}\,dt \ll(m+2)^D(\log T)^DE_j
\left(Te^{-(c/2)2^{3m/4}}+T^{3/4}\right).
\end{align*}
In this range $2^{3m/4}>(\log\log T)^{3/2}$.  The exponential terms are
summable, and after increasing the lower bound on $T$ we have
\begin{equation*}
T^{3/4}(\log T)^D(\log\log T)^{D+1}
\le T\exp\left(-\frac c8(\log\log T)^{3/2}\right).
\end{equation*}
The sums over these $m$ are therefore respectively
\begin{equation*}
\ll T(\log T)^D
\exp\left(-\frac c8(\log\log T)^{3/2}\right),
\qquad
\ll T(\log T)^DE_j
\exp\left(-\frac c8(\log\log T)^{3/2}\right).
\end{equation*}
On $\mathcal R_*$ we have $|R(t)|\le\sum_m2^{-m/10}\ll1$, so Lemmas
\ref{first-main-lemma} and \ref{second-main-lemma} apply.  Finally,
\eqref{eq:square-small-good} and \eqref{eq:square-small-bad} are summable in
$m$.  We increase the lower bound on $T$ once more so that
\begin{equation*}
(D+C_B)\log\log T+\log C_B
\le\frac c{16}(\log\log T)^{3/2}.
\end{equation*}
Lemma~\ref{B-poly-bounds} then absorbs the remaining powers of $\log T$.
This proves both assertions.
\end{proof}
	
\end{document}